\documentclass{amsart}

\usepackage{amsmath,amssymb,verbatim}
\usepackage{amsthm}
\usepackage{enumerate}
\usepackage{url}
\usepackage{mathtools}
\usepackage{tikz-cd}

\newtheorem{theorem}{Theorem}[section]
\newtheorem{lemma}[theorem]{Lemma}
\newtheorem{corollary}[theorem]{Corollary}
\newtheorem{proposition}[theorem]{Proposition}
\newtheorem{fact}[theorem]{Fact}

\theoremstyle{definition}
\newtheorem{definition}[theorem]{Definition}
\newtheorem{example}[theorem]{Example}
\newtheorem{remark}[theorem]{Remark}

\def\seq{\subseteq}
\def\nv{\text{-}}
\def\inv{^{\text{-}1}}

\def\cU{\mathcal{U}}

\def\cF{\mathcal{F}}

\def\cL{\mathcal{L}}

\def\N{\mathbb{N}}

\def\R{\mathbb{R}}

\def\T{\mathbb{T}}
\def\Z{\mathbb{Z}}
\def\C{\mathbb{C}}

\def\Stab{\operatorname{Stab}}

\def\tp{\operatorname{tp}}
\def\GL{\operatorname{GL}}

\def\Th{\operatorname{Th}}
\def\bk{\boldsymbol{k}}
\def\bone{\boldsymbol{1}}

\newcommand{\clqed}{\hfill$\dashv_{\text{\scriptsize{claim}}}$}

\newcommand{\dotminus}{ 
\!\!\buildrel\textstyle~.\over{\hbox{ 
\vrule height3pt depth0pt width0pt}{\smash-} 
}}

\newcommand{\miff}{\makebox[.4in]{$\Leftrightarrow$}}
\newcommand{\mand}{\makebox[.4in]{and}}

\newcommand{\Gen}[2]{\mathcal{M}_{#1}}

\def\U{\textnormal{U}}
\def\T{\textnormal{T}}

\AtBeginDocument{%
   \def\MR#1{}
}

\title{An analytic version of stable arithmetic regularity}

\date{June 17, 2024}

\author[G. Conant]{Gabriel Conant}
\author[A. Pillay]{Anand Pillay}

\thanks{Partially supported by   NSF grants DMS-1855503, DMS-2204787 (Conant) and DMS-1665035, DMS-1760212, DMS-2054271 (Pillay)}

\address{Department of Mathematics\\
The Ohio State University\\
Columbus, OH 43201\\
 USA}
\email{conant.38@osu.edu}

\address{Department Mathematics\\
University of Notre Dame\\
Notre Dame, IN 46556\\
 USA}
\email{apillay@nd.edu}

\begin{document}

\begin{abstract}
We prove a structure theorem for stable functions on amenable groups, which extends the  arithmetic regularity lemma for stable subsets of finite groups. Given a group $G$, a function $f\colon G\to [\text{-}1,1]$ is called stable if the binary function $f(x\cdot y)$ is stable in the sense of continuous logic. Roughly speaking, our main result says that if $G$ is amenable, then any stable function on $G$ is almost constant on all translates of a unitary Bohr neighborhood in $G$ of bounded complexity. The proof uses ingredients from topological dynamics and continuous model theory. We also prove several  applications  which generalize results  in arithmetic combinatorics to nonabelian groups.
\end{abstract}

\maketitle

\section{Introduction}

In \cite{TeWo}, Terry and Wolf prove an arithmetic regularity lemma for ``stable" subsets of $(\Z/p\Z)^n$. Their result compares to Green's \cite{GreenSLAG} arithmetic regularity lemma (in $(\Z/p\Z)^n$) in direct analogy to how Malliaris and Shelah's \cite{MaSh} regularity lemma for stable graphs   compares to  Szemer\'{e}di's \cite{SzemRL} famous lemma for all finite graphs. 
Shortly after \cite{TeWo}, a non-quantitative generalization to arbitrary  finite groups was  proved by the authors and Terry \cite{CPT}, followed by a quantitative version for finite abelian groups by Terry and Wolf \cite{TeWo2}. A quantitative generalization to all finite groups was later obtained  by the first author \cite{CoQSAR}. 

The main goal of this paper is to extend the study of arithmetic regularity to stable functions  (see Definitions \ref{def:stable} and \ref{def:stableG}). While this task has natural model-theoretic interest due to the connection to continuous logic, we will also combine our main result with the uniform stability of inner products in Hilbert spaces to obtain several applications in arithmetic combinatorics related to Bogolyubov's Lemma and the Croot-Sisask Lemma. Thus our results establish stable arithmetic regularity  as an anchor for several major facets of the interaction between model theory and combinatorics that has developed over the past fifteen years or so. Further details are given at the end of this introduction and in Section \ref{sec:app}.

To motivate the statement of our main result, we  recall Malliaris and Shelah's regularity lemma \cite{MaSh} which, roughly stated, says that if $E\seq V\times V$ is a stable binary relation (with $V$ finite), then there is a partition $\{V_i\}$ of $V$ of efficiently bounded size such that the relation induced by $E$ on each pair $V_i\times V_j$ is almost complete or almost empty. In \cite{CCP}, the authors and Chavarria proved a natural qualitative generalization of this statement for a stable binary function $f\colon V\times V\to [0,1]$,  in which the conclusion is that $f$ is almost constant on almost all of each pair $V_i\times V_j$ (but without an explicit bound on the size of the partition).

For a finite group $G$, the results  from \cite{TeWo,TeWo2,CPT,CoQSAR} can be interpreted as saying that if $A\seq G$ is a stable set, then  the relation on $G\times G$ defined by $xy\in A$ admits a partition as in \cite{MaSh}, but where the pieces are  the  cosets of a normal subgroup  of $G$. Following the analogy, one then might guess that a stable function on a finite group is almost constant on almost all of each coset of some such subgroup. However, this turns out to be false (see Example \ref{ex:ZpZ}) and, as is often the case in arithmetic combinatorics, we will need to abandon the graph-theoretic focus on partitions and replace subgroups with Bohr neighborhoods. The model-theoretic explanation for this is that tame arithmetic regularity results correspond to  ``domination" by a certain compact group $K$ (along the lines of $G/G^{00}$). Arithmetic regularity in terms of coset partitions then arises from situations where $K$ is profinite.
For the case of stable sets, $K$ turns out to be a closed subgroup of a topological semigroup  on a type space, which is of course profinite (in  classical first-order logic). For stable functions, we will  obtain $K$ from a type space in a similar way, but in the context of continuous logic where type spaces need not be profinite.

We now describe  our main result, which will be formulated in the more general setting of (discrete) amenable groups. Recall that a group $G$ is \textbf{amenable} if  there is a left-invariant finitely additive probability measure on the Boolean algebra of subsets of $G$. For brevity, we will simply say \emph{left-invariant measure} when working in this context. We will also need the following  terminology. Let $G$ be a group and fix a function $f\colon G\to \R$. Given $B\seq G$, we say $f$ is \textbf{$\epsilon$-constant on $B$} if $|f(x)-f(y)|<\epsilon$ for all $x,y\in B$. If $G$ is amenable with a fixed left-invariant measure $\mu$,  then we say $f$ is \textbf{$\zeta$-almost $\epsilon$-constant on $B$} if it is $\epsilon$-constant on a set $B'\seq B$ with $\mu(B')\geq \mu(B)-\zeta$.\footnote{Remark \ref{rem:almost} explains why we have not chosen to write $(1-\zeta)\mu(B)$ here instead.} Finally, given a function $\bk\colon \R^+\to \Z^+$, we say that  $f$ is \textbf{$\bk$-stable} if for all $\epsilon>0$, there do not exist $a_1,\ldots,a_{\bk(\epsilon)},b_1,\ldots,b_{\bk(\epsilon)}\in G$ such that $|f(a_ib_j)-f(a_jb_i)|\geq\epsilon$ for all $1\leq i<j\leq\bk(\epsilon)$. 

\begin{theorem}[main result]\label{thm:SARF}
Fix  $\epsilon>0$, $\zeta\colon\R^+\times \N\to\R^+$, and $\bk\colon \R^+\to\Z^+$. Suppose $G$ is an amenable group with left-invariant measure $\mu$, and $f\colon G\to[\nv 1,1]$ is $\bk$-stable. Then there is a $(\delta,\U(n))$-Bohr neighborhood $B$ in $G$, with $\delta\inv,n\leq O_{\bk,\zeta,\epsilon}(1)$, such that $f$ is $\zeta(\delta,n)$-almost $\epsilon$-constant on all translates of $B$. 
\end{theorem}

In the previous statement, $\U(n)$ denotes the unitary group of degree $n$, and a $(\delta,\U(n)$)-Bohr neighborhood is a homomorphic preimage of the identity neighborhood of radius $\delta$ in $\U(n)$. See Definition \ref{def:Bohr} for details. Note also that the statement involves two error parameters, one of which is itself a \emph{function} of the parameters associated to the unitary Bohr neighborhood $B$. This feature will be needed for the applications of Theorem \ref{thm:SARF} in Section \ref{sec:app}. But for the sake of exposition, we state the following special case, which follows immediately from Theorem \ref{thm:SARF} and Proposition \ref{prop:Bohr}$(a)$ (see also Remark \ref{rem:almost}). 

\begin{corollary}\label{cor:SARF}
Fix  $\epsilon>0$ and $\bk\colon \R^+\to\Z^+$. Suppose $G$ is an amenable group with left-invariant measure $\mu$, and $f\colon G\to[\nv 1,1]$ is $\bk$-stable. Then there is a $(\delta,\U(n))$-Bohr neighborhood $B$ in $G$, with $\delta\inv,n\leq O_{\bk,\epsilon}(1)$, such that $f$ is $\epsilon\mu(B)$-almost $\epsilon$-constant on all translates of $B$. 
\end{corollary}

 Roughly speaking, Theorem \ref{thm:SARF} says that a stable function on an amenable group behaves similarly to a continuous function on a compact group. Since continuous  functions are  canonical examples of stable functions, one can thus interpret Theorem \ref{thm:SARF} as a ``rigidity theorem" in the sense of \cite{TaoICM}. See Section \ref{sec:preG} for further discussion along these lines.

To prove Theorem \ref{thm:SARF}, we will first obtain a corresponding model-theoretic result in continuous logic, and then apply an ultraproduct argument. The model-theoretic companion result is given in Section \ref{sec:mainMT} (see Theorem \ref{thm:mainUP}), and concerns a metric structure $M$ with a sort $G$ for a group, along with a left-invariant stable formula $\varphi(x,z)$ (with $x$ of sort $G$). The key consequence of stability  is weak almost periodicity of the type space $S_\varphi(M)$ as a $G(M)$-flow. By work of Ellis and Nerurkar \cite{EllNer} in topological dynamics, it follows that the Ellis semigroup of $S_\varphi(M)$ has a unique minimal subflow $K$, which is a compact Hausdorff group. In Lemma \ref{lem:ESid}, we provide  a precise description of the Ellis semigroup of $S_\varphi(M)$ as a certain type space, which then allows us to view $K$ as a definable compactification of $G$. This also opens the route to  the main model-theoretic result (Theorem \ref{thm:mainUP}), which says roughly that any $\varphi$-formula $\theta(x)$ is almost constant  on all translates of a preimage of an identity neighborhood in $K$. Finally, we use the Peter-Weyl Theorem to replace $K$ with a unitary group (see Corollary \ref{cor:mainUP}).  All of this sets the stage for proving Theorem \ref{thm:SARF} with ultraproducts, although in this case the transfer argument is not nearly as straightforward as similar situations in classical logic (e.g., \cite{CPT}). For example, we will use recent work with Hrushovski \cite{CHP} to facilitate the use of Bohr neighborhoods in ultraproducts. This step of the proof represents a key use of the amenability assumption through a result of Kazhdan \cite{Kazh} on approximate homomorphisms.  See the start of Section \ref{sec:main} for further discussion. 

After the proof of Theorem \ref{thm:SARF}, we  state some refinements in the special cases of groups of bounded exponent, abelian groups, and torsion groups, where unitary Bohr neighborhoods can be replaced by subgroups and/or Bohr neighborhoods defined from a finite-dimensional  torus (see Corollaries  \ref{cor:abelian}, \ref{cor:exponent}, and \ref{cor:torsion}). Then in Section \ref{sec:app} we will turn to   applications of Theorem \ref{thm:SARF} which exploit  the fact that  the continuous theory of Hilbert spaces is stable. In particular, this fact implies that any function on an amenable group obtained as a convolution of two functions is $\bk$-stable for some absolute  $\bk$ (see Definition \ref{def:conv} and  Corollary \ref{cor:conv-stable}). Combined with Theorem \ref{thm:SARF}, we have the following conclusion.

\begin{corollary}\label{cor:conv-regularity}
Fix  $\epsilon>0$ and $\zeta\colon\R^+\times \N\to\R^+$. Suppose $G$ is an amenable group with left-invariant measure $\mu$, and $f,g\colon G\to[\nv 1,1]$ are arbitrary functions. Then there is a $(\delta,\U(n))$-Bohr neighborhood $B$ in $G$, with $\delta\inv,n\leq O_{\zeta,\epsilon}(1)$, such that $f\ast g$ is $\zeta(\delta,n)$-almost $\epsilon$-constant on all translates of $B$. 
\end{corollary}

In Section \ref{sec:app}, we use this  to give a short proof of Bogolyubov's Lemma for amenable groups (proved recently in \cite{CHP}), a ``two-set" generalization of Bogolyubov for finite groups (which connects to work of Gowers \cite{GowQRG} on quasirandom groups), and a non-quantitative version of the Croot-Sisask Lemma for amenable groups. 

Finally, in Section \ref{sec:more}, we will return to the general model-theoretic setting of Section \ref{sec:mainMT}, and prove some further results about generic types, stabilizers, and connected components. For the most part, these will be natural generalizations of known results in classical logic (e.g., from \cite{CPT} and \cite{CoLSGT}), although we will make note of some crucial differences. 

\subsection*{Acknowledgements} We thank Caroline Terry for helpful discussions and Julia Wolf for several invaluable comments on an earlier draft.

\section{Preliminaries}\label{sec:prelim}

\subsection{Notation}
Given $r,s\in\R$ and $\epsilon>0$, we write $r\approx_\epsilon s$ to denote $|r-s|\leq \epsilon$.

\subsection{Bohr neighborhoods}

\begin{definition}\label{def:Bohr}
Let $K$ be a group equipped with a metric. Fix a (discrete) group $G$ and a real number $\delta>0$ with $\delta\leq\textnormal{diam}(K)$. Then a \textbf{$(\delta,K)$-Bohr neighborhood in $G$} is a set of the form $\tau\inv(U)$, where $\tau\colon G\to K$ is a group homomorphism and $U$ is the open identity neighborhood in $K$ of radius $\delta$.
\end{definition}

In the above context, $K$ is typically a compact group metrizable by the chosen metric. We will focus on the following cases:
\begin{enumerate}[\hspace{5pt}$\ast$]
\item $K$ is  the  unitary group $\U(n)$ of degree $n$ with the metric induced by the operator norm on $\textnormal{GL}_n(\C)$, or
\item $K$ is the real $n$-dimensional torus $\T(n)$ with the metric induced from $\U(n)$.
\end{enumerate}
In the second case, we view $\T(n)$ as the maximal torus consisting of diagonal matrices. Thus the metric on $\T(n)$ is the Cartesian product of the complex distance metric on the unit circle. 

Note that the above metric  on $\U(n)$ is bi-invariant. Consequently, if $B$ is a $(\delta,\U(n))$-Bohr neighborhood in a group $G$, then $B=B\inv$ and $gB=Bg$ for any $g\in G$. 
The following additional  facts  are proved in \cite[Section 5]{CHP}.

\begin{proposition}\label{prop:Bohr}
Let $B$ be a $(\delta,\U(n))$-Bohr neighborhood in a group $G$.
\begin{enumerate}[$(a)$]
\item $G$ can be covered by $(c/\delta)^{n^2}$ translates of $B$, where $c>0$ is an absolute constant.
\item If $G$ is abelian then $B$ is a $(\delta,\T(n))$-Bohr neighborhood.
\item If $G$ has exponent $r$ and $\delta\leq O_r(1)$, then $B$ is a normal subgroup of $G$.
\end{enumerate}
\end{proposition}

\begin{remark}\label{rem:almost}
Suppose $G$ is amenable with left-invariant measure $\mu$. Then Proposition \ref{prop:Bohr}$(a)$ implies that any $(\delta,\U(n))$-Bohr neighborhood in $G$ has measure at least $(\delta/c)^{n^2}$. Now recall that a function $f\colon G\to \R$ is  \emph{$\zeta$-almost $\epsilon$-constant} on a set $B\seq G$ if it is $\epsilon$-constant on some $B'\seq  B$ with $\mu(B')\geq\mu(B)-\zeta$. This statement loses potency if $\zeta$ is large relative to $\mu(B)$. However our main result (Theorem \ref{thm:SARF}) is formulated with $\zeta$ a \emph{function}  of the parameters $\delta$ and $n$ associated to a unitary Bohr set. So, for example, suppose $f\colon G\to \R$ is $\gamma(c/\delta)^{n^2}$-almost $\epsilon$-constant on a $(\delta,\U(n))$-Bohr neighborhood $B\seq G$ for some $\gamma,\epsilon>0$. Then $f$ is $\epsilon$-constant on some $B'\seq B$ with $\mu(B')\geq (1-\gamma)\mu(B)$. In other words, requiring $\mu(B')\geq (1-\zeta)\mu(B)$ in the definition of ``$\zeta$-almost" would not affect the statement of Theorem \ref{thm:SARF}. That being said, we have chosen to not use this latter formulation in the actual definition in order to remove irrelevant (and distracting) steps in certain arguments.
\end{remark}

\subsection{Stable binary functions on sets}

\begin{definition}\label{def:stable}
Let $f\colon V\times W\to \R$ be a function, where $V$ and $W$ are arbitrary sets. Given an integer $k\geq 1$ and a real number $\epsilon>0$, we say that $f$ is \textbf{$(k,\epsilon)$-stable} if there do not exist $a_1,\ldots,a_k\in V$ and $b_1,\ldots,b_k\in W$  such that 
\[
|f(a_i,b_j)-f(a_j,b_i)|\geq\epsilon \text{ for all $1\leq i<j\leq k$}.
\]
Given $\bk\colon \R^+\to \Z^+$, we say that $f$ is \textbf{$\bk$-stable} if it is $(\bk(\epsilon),\epsilon)$-stable for all $\epsilon>0$.
\end{definition}

For a binary relation $E\seq V\times W$, stability of the indicator function $\bone_E$ recovers the standard definition of stability for $E$, up  to a uniform change in the parameters. See \cite[Remark 1.2]{CCP} for precise details. 

We will now discuss two families of stable functions, which can both be viewed as coming from stability of certain continuous theories. 

\begin{example}\label{ex:cont-stable}
Suppose $X$ and $Y$ are compact spaces and $f\colon X\times Y\to \R$ is a continuous function. Then $f$ is $\bk$-stable for some $\bk\colon \R^+\to\Z^+$.
\end{example}

When $X$ and $Y$ are metrizable, this  is the continuous logic analogue of the trivial fact that any binary relation on finite sets is stable. That being said, a direct proof is still an informative exercise (which we leave to the reader). Moreover, as discussed in the next subsection, this example will provide some useful intuition for our main arithmetic regularity results.  

The next example is much less trivial, but also well known (and previously discussed in \cite[Section 1]{CCP}).

\begin{fact}\label{fact:Hilbert}
There is a function $\bk\colon\R^+\to\Z^+$ such that if $B$ is the unit ball of a Hilbert space, then the inner product is $\bk$-stable as a function from $B\times B$ to $[\nv 1,1]$.
\end{fact}

While this fact follows from stability of Hilbert spaces in continuous logic, which is proved in \cite{BBHU}, it is also evident from much earlier literature.\footnote{Thanks to Ward Henson for communicating the following summary.} In particular, a slightly weaker version of stability (known as ``stability in a model") can be deduced for the inner product from a result of Grothendieck \cite{GroWAP}. Alternatively, stability (in a model) of the function $\|x+y\|$ is proved by Krivine and Maurey \cite{KrMau}, which then passes to the inner product using the parallelogram law. In either case, to obtain a uniform stability function $\bk$, one can  use the fact that Hilbert spaces are closed under ultrapowers, which was first established in Krivine's thesis. 

One takeaway from Fact \ref{fact:Hilbert} is that stable functions arise quite naturally across mathematics and perhaps are more common than stable graphs. Indeed, this fact has previously led to remarkable applications of stability in many areas. Further discussion can be found in \cite[Section 1]{CCP}.

\subsection{Stable unary functions on groups}\label{sec:preG}

Stability can be viewed as a property of a subset $A$  of a group $G$ by considering the binary relation $x\cdot y\in A$ on $G$. This is the context of the arithmetic regularity results in \cite{TeWo, TeWo2, CPT, CoQSAR}. For functions on groups, we  proceed analogously.

\begin{definition}\label{def:stableG}
Fix $\bk\colon\R^+\to\Z^+$. Given a group $G$, we say that a function $f\colon G\to \R$ is $\bk$-stable if $f(x\cdot y)\colon G\times G\to \R$ is $\bk$-stable.
\end{definition}

As previously observed, if $A$ is a stable subset of a group $G$, then $\bone_A$ is a stable function on $G$. For a different kind of example, it follows from Example \ref{ex:cont-stable} that if 
$G$ is a compact Hausdorff group, then any continuous function $f\colon G\to \R$ is $\bk$-stable for some $\bk$. As alluded to in the introduction,  this connects to our main results via the following standard exercise in topological groups.

\begin{proposition}\label{prop:continuous}
Let $G$ be a compact Hausdorff group. Then a function $f\colon G\to \R$ is continuous if and only if for every $\epsilon>0$, there is an open identity neighborhood $U\seq G$ such that $f$ is $\epsilon$-constant on all translates of $U$. 
\end{proposition}

In light of the previous fact, Theorem  \ref{thm:SARF} says that if a function $f$ on an amenable group  behaves like a continuous function on a compact  group in the sense of being stable, then $f$ is structurally similar to such a  function.

\begin{remark}
The previous discussion can be made more precise through the use of the Bohr topology on a  group $G$ (i.e., the possibly non-Hausdorff topology induced from the Bohr compactification of $G$). Given $f\colon G\to \R$, if for all $\epsilon>0$ there are $\delta,n$  such that $f$ is $\epsilon$-constant on all translates of a $(\delta,\U(n))$-Bohr neighborhood in $G$, then $f$ is uniformly continuous with respect to the Bohr topology on $G$. So Theorem \ref{thm:SARF} says that stable functions exhibit an approximate form of uniform Bohr-continuity. Thus the role of uniformly Bohr-continuous functions on $G$  in arithmetic regularity for stable functions is roughly analogous to the role of unions of cosets of subgroups of $G$ in arithmetic regularity for stable sets. 
\end{remark}

Next we use Fact \ref{fact:Hilbert} to identify another example of stable functions, which will play a key role in our applications in Section \ref{sec:app}.
Recall that if $G$ is amenable, then any left-invariant measure $\mu$ on $G$ uniquely determines a positive linear functional $\int f\,d\mu$ on bounded real-valued functions $f$ on $G$, which satisfies  $\mu(A)=\int \bone_A\,d\mu$ for all $A\seq G$. See  \cite[pp. 2-3]{Patbook} for details. 

\begin{definition}\label{def:conv}
Let $G$ be an amenable group with left-invariant measure $\mu$. Given bounded functions $f,g\colon G\to \R$, the \textbf{convolution of $f$ and $g$} is the function $f\ast g\colon G\to \R$ such that for $x\in G$,
\[
(f\ast g)(x)=\int f(t)g(t\inv x)\, d\mu(t).
\]
\end{definition}

The next corollary is essentially the same as \cite[Corollary 1.4$(ii)$]{CCP}, but in a slightly different setting. 

\begin{corollary}\label{cor:conv-stable}
Let $\bk\colon \R^+\to\Z^+$ be as in Fact \ref{fact:Hilbert}. Suppose $G$ is an amenable group and  $f,g\colon G\to [\nv 1,1]$ are arbitrary functions. Then $f\ast g$ is $\bk$-stable.
\end{corollary}
\begin{proof}
Let $\mu$ be the implied measure on $G$ from the statement of the corollary, which we view as a regular Borel probability measure on $\beta G$. Consider the Hilbert space $H=L^2(\beta G,\mu)$ with inner product $\langle f,g\rangle=\int fg\,d\mu$, and let $B$ denote the unit ball. Note that any  function $f\colon G\to [\nv 1,1]$ extends uniquely to a continuous function on $\beta G$ in $B$. Given  $f,g\colon G\to  [\nv 1,1]$, define $f_1,g_2\colon G\to B$ such that $f_1(x)=f(x t)$ and $g_2(y)=g(t\inv y)$. Then $(f\ast g)(xy)=\langle f_1(x), g_2(y)\rangle$. So $f\ast g$ is $\bk$-stable (as a function on the group $G$) by Fact \ref{fact:Hilbert}.
\end{proof}

The next example uses Corollary \ref{cor:conv-stable} to show that one cannot replace Bohr neighborhoods with subgroups in Theorem \ref{thm:SARF} (unlike the case of stable sets).

\begin{example}\label{ex:ZpZ}
Let $G=\Z/p\Z$ for some prime $p>2$, and let $A=\{0,1,\ldots,\frac{p-1}{2}\}$. Define $f\colon G\to [0,1]$ such that $f(x)=\frac{1}{p}|A\cap (x+A)|$. Then  $f=\bone_A\ast \bone_{\nv A}$, and so $f$ is $\bk$-stable where $\bk$ is as in Fact \ref{fact:Hilbert}.  However, as $p$ increases, $f$ takes values arbitrarily close to $0$ and  to $\frac{1}{2}$ on cyclic intervals of positively bounded measure. So for  sufficiently small $\epsilon$, $f$ is not $\epsilon$-almost $\epsilon$-constant on $G$ (which is the only subgroup of index independent of $p$).
\end{example}

On the other hand, if one restricts to groups of a fixed finite exponent, then such a strengthening of Theorem \ref{thm:SARF} does follow; see Corollary \ref{cor:exponent}.

\subsection{Topological dynamics}

Let $G$ be a group. A \textbf{$G$-flow} is a compact Hausdorff space $S$ together with a left action of $G$ by homeomorphisms. For $g\in G$, let $\breve{g}$ denote the corresponding homeomorphism of $S$ (following the notation from Glasner's survey \cite{Glas-surv}). As shown by Ellis \cite{EllLTD}, the closure of  $\{\breve{g}:g\in G\}$ in the space of functions from $S$ to $S$ is a semigroup under  composition, called the \emph{Ellis semigroup} $E(S)$ of $S$.\footnote{We caution the reader that \cite{EllLTD} works in the context of right actions (as does \cite{EllNer} cited in Theorem \ref{thm:EN}), whereas we use left actions (as does \cite{Glas-surv}).} Note that $E(S)$ is itself a $G$-flow under the natural action of $G$. One can also check that for a fixed $\sigma\in E(S)$, the right composition map $\tau\mapsto \tau\circ\sigma$ is continuous; so $E(S)$ is a \emph{right topological semigroup}. The following fact (which is a standard exercise) says that the Ellis semigroup operation is idempotent.

\begin{fact}\label{fact:E=EE}
Let $S$ be $G$-flow. Then for any $\sigma\in E(S)$, the left composition map $L_\sigma\colon \tau\mapsto \sigma\circ\tau$ is in $E(E(S))$. Moreover, $\sigma\mapsto L_\sigma$ is an isomorphism  (of $G$-flows and right topological semigroups) from $E(S)$ to $E(E(S))$.
\end{fact}

The next result is the key ingredient from topological dynamics needed for our proof of Theorem \ref{thm:SARF}. 

\begin{theorem}[Ellis \& Nerurkar \cite{EllNer}]\label{thm:EN}
Let $S$ be a $G$-flow with Ellis semigroup $E$, and assume that every function in $E$ is continuous. Then:
\begin{enumerate}[$(a)$]
\item $E$ has a unique minimal subflow $K$.
\item $K$ is the unique left ideal in $(E,\circ)$.
\item $(K,\circ)$ is a compact Hausdorff group.
\item The identity in $K$ commutes with every element of $E$.
\item$E$ is \textbf{uniquely ergodic}, i.e., there is a unique $G$-invariant regular Borel probability measure on $E$.
\end{enumerate}
\end{theorem}

A $G$-flow satisfying the assumptions of the previous theorem is called \textbf{weakly almost periodic} (although this is not the official definition; see \cite{EllNer} for further details).  It is worth noting that the proof of Theorem \ref{thm:EN}  relies on the same result of Grothendieck \cite{GroWAP} mentioned above in the context of Fact \ref{fact:Hilbert}. As shown by Ben Yaacov \cite{BYGro}, this result can also be used to give a quick proof of Theorem \ref{thm:SFT} below, which is the main model-theoretic ingredient in our proof of Theorem \ref{thm:SARF}. 

\subsection{Continuous logic and stability}\label{sec:CL}

We will work in the standard framework of continuous logic as in the monograph \cite{BBHU}. The reader is also referred to \cite[Section 2]{CCP} for further details of certain aspects not directly discussed in \cite{BBHU}. Throughout this subsection, $\cL$ denotes a first-order language in continuous logic. 

\begin{definition}\label{def:formula}
Let $M$ be an $\cL$-structure  and let $\Phi$ be a collection of partitioned $\cL$-formulas of the form $\varphi(x,y)$ where $x$ and $y$ are fixed tuples of variables. Then a \textbf{$\Phi$-formula over $M$} is a continuous function from $S_\Phi(M)$ to $\R$, where $S_\Phi(M)$ is the compact Hausdorff space of complete $\Phi$-types over $M$.
\end{definition}

See \cite[Section 2]{CCP} and \cite[Section 6]{BYU} for further details in the case when $\Phi$ is a single formula $\varphi(x,y)$. Recall that in general, a $\Phi$-formula is  built from formulas in $\Phi$ using an \emph{infinitary} continuous connective (see \cite[Fact 6.1]{BYU}). For this reason, what we call $\Phi$-formula would be referred to in \cite{BYU} as a ``$\Phi$-predicate", with ``formula"  reserved for the case of  finitary connectives. Note that this discrepancy does not arise in classical logic, where a $\Phi$-formula (as we have defined it) is always given by a finite Boolean combination of instances from $\Phi$.

We will be working with Keisler measures in continuous logic, viewed either as linear functionals on formulas or as regular Borel probability measures on type spaces. See \cite[Section 3.4]{BYtop} and \cite[Section 2.3]{CCP} for details. However, we will diverge from these sources in that we will use the same notation for  functionals and measures. To clarify, let $M$ be an $\cL$-structure, and suppose $\mu$ is a Keisler measure over $M$ in variables $x$, i.e., a regular Borel probability measure on $S_x(M)$. Given an $\cL$-formula $\theta(x)$ over $M$, we let $\mu(\theta(x))$ denote $\int_{S_x(M)}\theta(x)\, d\mu$. Given an $\cL$-formula $\theta(x)$ and a Borel set $B\seq \R$, we write $\mu(\theta(x)\in B)$ for $\mu(\theta\inv(B))$. We will also use this notation in the local situation of \emph{Keisler $\Phi$-measures over $M$}.

We now recall the definition of stability for formulas in continuous logic.\footnote{This definition is formulated specifically for our purposes, but one can easily check it is equivalent to the standard definition from other sources (e.g., \cite{BYU}).}

\begin{definition}
Let $T$ be a complete $\cL$-theory and fix an $\cL$-formula $\varphi(x,y)$.
\begin{enumerate}[$(a)$]
\item  Given $k\geq 1$ and $\epsilon>0$, we say $\varphi(x,y)$ is \textbf{$(k,\epsilon)$-stable (in $T$)} if for every $M\models T$, the function $\varphi\colon M^x\times M^y\to \R$ is $(k,\epsilon)$-stable. 
\item We say that $\varphi(x,y)$ is \textbf{stable (in $T$)} there is some $\bk\colon \R^+\to\Z^+$ such that for all $\epsilon>0$, $\varphi(x,y)$ is $(k(\epsilon),\epsilon)$-stable (in $T$). In this case, we also say $\varphi(x,y)$ is \textbf{$\bk$-stable (in $T$)}.
\end{enumerate}
\end{definition}

The following is a basic exercise (see also the proof of \cite[Lemma 5.5]{CCP}). 

\begin{proposition}\label{prop:stableup}
Let $(M_i)_{i\in I}$ be a family of $\cL$-structures and let $\cU$ be an ultrafilter on $I$. Fix an $\cL$-formula $\varphi(x,y)$ and assume that for some $k\geq 1$ and $\epsilon>0$, $\varphi(x,y)$ is $(k,\epsilon)$-stable in $\Th(M_i)$ for all $i\in I$. Then for any $\epsilon'>\epsilon$, $\varphi(x,y)$ is $(k,\epsilon')$-stable in $\Th(\prod_{\cU}M_i)$.
\end{proposition}

 The next theorem is an important result from the foundations of stability theory (see \cite[Chapter II]{Shbook} and \cite[Chapter 1]{PiGST}). The generalization to continuous theories was done by  Ben Yaacov and Usvyatsov (see Propositions 7.6 and 7.16 in \cite{BYU}). Given an $\cL$-formula $\varphi(x,y)$, we let $\varphi^*(y,x)$ denote the same formula, but with the roles of object and parameter variables exchanged.

\begin{theorem}\label{thm:SFT}
Let $T$ be a complete $\cL$-theory and suppose $\varphi(x,y)$ is an $\cL$-formula that is stable in $T$. 
\begin{enumerate}[$(a)$]
\item For any $M\models T$, if $p\in S_\varphi(M)$ then the map $b\mapsto \varphi(p,b)$ from $M^y$ to $\R$ is given by a  $\varphi^*$-formula over $M$, which we denote by $d^\varphi_p(y)$.
\item For any $M\models T$, if $p\in S_\varphi(M)$ and $q\in S_{\varphi^*}(M)$, then $d^\varphi_p(q)=d^{\varphi^*}_q(p)$.
\end{enumerate}
\end{theorem}

We will frequently use the notation $d^\varphi_p(y)$ in situations where $p$ is a type in a  larger fragment of formulas that includes all $\varphi$-formulas (over some fixed model). In such cases, it is understood that $d^\varphi_p(y)$ means $d^{\varphi}_{p|_\varphi}(y)$.

\section{Main model-theoretic result}\label{sec:mainMT}
Throughout this section, we let $T$ be a continuous $\cL$-theory with a sort $G$ expanding a group. We also fix a model $M\models T$ and identify $G$ with $G(M)$. Let $x$ be of sort $G$, and let $\varphi(x,z)$ be an $\cL$-formula that is \emph{left-invariant}, i.e., for any $b\in M^z$ and $g\in G$ there is  $c\in M^z$ such that $\varphi(gx,b)=\varphi(x,c)$.  From this assumption it follows that the type space $S_\varphi(M)$ is a $G$-flow under the natural action.

 \textbf{For the entirety of this section, we assume $\varphi(x,z)$ is stable in $T$.}  Our main goal is Theorem \ref{thm:mainUP} below, which provides an arithmetic regularity statement for $\varphi$-formulas  in terms of a canonical definable compactification of $G$. The rough strategy will  follow the methods in \cite{CoLSGT} using topological dynamics. In particular, we will use Theorem \ref{thm:SFT} to give a concrete description of the Ellis semigroup of $S_\varphi(M)$,  and  then apply the results of Ellis and Nerurkar in Theorem \ref{thm:EN}.
 
Given $b\in M^z$, the formula $\varphi(y\cdot x,b)$ is itself stable by left-invariance and stability of $\varphi(x,z)$. So by Theorem \ref{thm:SFT}$(a)$, for every $b\in M^z$ and $p\in S_\varphi(M)$, we have a defining formula $d^b_p(y)\coloneqq d^{\varphi(y\cdot x,b)}_p(y)$.

  Now let  $\varphi^\sharp(x;y,z)$ denote the bi-invariant formula $\varphi(x\cdot y,z)$. \textit{We emphasize that this formula need not  be stable} (see \cite[Example 3.7]{CPpfNIP} and \cite[Example 5.11]{CoLSGT}). Nevertheless, given  $b\in M^z$ and $p\in S_{\varphi^\sharp}(M)$, we  have a $\varphi^\sharp$-formula $d^b_p(y)$ as above. Note that for any $b\in M^z$ and any $p$ in  $S_\varphi(M)$ or $S_{\varphi^\sharp}(M)$, the formula $d^b_p(y)$ is a $\varphi^\sharp$-formula over $M$ by Theorem \ref{thm:SFT}$(a)$.

  For the rest of this section,  let $S=S_{\varphi}(M)$ and $S^\sharp=S_{\varphi^\sharp}(M)$.

\begin{lemma}\label{lem:ESid}
$S^\sharp$ is isomorphic as a $G$-flow to the Ellis semigroup  of $S$. Moreover:
\begin{enumerate}[$(i)$]
\item The induced semigroup operation $\ast$ on $S^\sharp$ is given by $\varphi((p\ast q)g,b)=d^{b}_{qg}(p)$ where $p,q\in S^\sharp$, $g\in G$, and $b\in M^z$.
\item For any $p\in S^\sharp$, the map $q\mapsto p\ast q$ from $S^\sharp$ to $S^\sharp$ is continuous.
\end{enumerate}
\end{lemma}
\begin{proof}
Let $X$ be the space of real-valued functions on $M^z$. Note that $S$ can be identified as a closed subset of $X$. So $S^S$ is closed in $X^S$, and  $E(S)$ is the closure of $\{\breve{g}:g\in G\}$ in $X^S$ (recall $\breve{g}$ denotes the $g$-action map on $S$). Define a function $\Sigma \colon p\mapsto \sigma_p$ from $S^\sharp$ to $X^S$ so that, given $q\in S$ and $b\in M^z$, $\sigma_p(q)[b]=d^{b}_q(p)$.

We first show $\Sigma$ is continuous. By definition of the topology on $X^S$, it suffices to consider an open set $W\seq X^S$ of the form $\{\sigma\in X^S:\sigma(q)[b]\in I\}$ for some fixed $q\in S$, $b\in M^z$, and open interval $I\seq \R$. For $p\in S^\sharp$, we have $\sigma_p\in W$ if and only if $d^b_q(p)\in I$, and hence $\Sigma\inv(W)$ is the open set in $S^\sharp$ determined by $d^b_q(x)\in I$. 

Now note that $\Sigma(\tp_{\varphi^\sharp}(g/M))=\breve{g}$ for any $g\in G$.
Since the realized types are dense in $S^\sharp$, and $\Sigma$ is  continuous  from the compact space $S^\sharp$ to the Hausdorff space $X^S$, it follows that $\Sigma(S^\sharp)$ is the closure of $\{\breve{g}:g\in G\}$, i.e., $\Sigma(S^\sharp)=E(S)$. In particular, if $p\in S^\sharp$ and $q\in S$ then $\sigma_p(q)$ is the unique type in $S$ satisfying $\varphi(\sigma_p(q),b)=d^b_q(p)$ for all $b\in M^z$.

Next we show $\Sigma$ is injective, and  thus a homeomorphism from $S^\sharp$ to $E(S)$ (since both spaces are compact Hausdorff). Fix $p\in S^\sharp$. Then for any $g\in G$ and $b\in M^z$, if $q=\tp_\varphi(g/M)$ then  $\varphi(pg,b)=d^b_{q}(p)=\varphi(\sigma_p(q),b)$
 (the first equality is clear when $p$ is a realized type in $S^\sharp$, and thus holds for all $p$ by continuity of $d^b_q(y)$).  
 Hence $p$ is uniquely determined by $\sigma_p$. 

Finally, we show that $\Sigma$ preserves the action of $G$. Fix $g\in G$, $q\in S$, and $b\in M^z$. We need to show $\varphi(\sigma_{gp}(q),b)=\varphi(g\sigma_p(q),b)$ for any $p\in S^\sharp$. Fix $c\in M^z$ such that $\varphi(gx,b)=\varphi(x,c)$. Then $d^b_q(gy)=d^c_q(y)$, and so for any $p\in S^\sharp$, we have 
\[
\varphi(\sigma_{gp}(q),b)=d^b_q(gp)=d^c_q(p)=\varphi(\sigma_p(q),c)=\varphi(g\sigma_p(q),b).
\] 

We have now shown that $\Sigma$ is the desired $G$-flow isomorphism from $S^\sharp$ to $E(S)$. It remains to verify  $(i)$ and $(ii)$. Statement $(i)$ is a straightforward calculation.  
For $(ii)$, fix $p\in S^\sharp$ and let $L_p$ denote $q\mapsto p\ast q$. Fix $g\in G$, $b\in M^z$, and an open interval $I\seq \R$, and consider the sub-basic open set $W=\{q\in S^\sharp:\varphi(qg,b)\in I\}$. Let $\psi(x,y)=\varphi(y\cdot x,b)$, and recall that $\psi(x,y)$ is stable. Applying Theorem \ref{thm:SFT}$(b)$ and $(i)$, we have that for $q\in S^\sharp$,
\[
L_p(q)\in W\miff \varphi((p\ast q)g,b)\in I\miff d^{\psi}_{qg}(p)\in I\miff d^{\psi^*}_p(qg)\in I.
\]
So $L_p\inv(W)$ is the open set in $S^\sharp$ determined by $d^{\psi^*}_p(xg)\in I$. 
\end{proof}

 \newpage

\begin{corollary}\label{cor:EN}
$~$
\begin{enumerate}[$(a)$]
\item $S^\sharp$ has a unique minimal subflow $K$.
\item $K$ is the unique left ideal in $(S^\sharp,\ast)$
\item $(K,\ast)$ is a compact Hausdorff group.
\item The identity in $K$ commutes with every element of $S^\sharp$.
\item There is a unique left-invariant $\varphi^\sharp$-Keisler measure on $G$.
\end{enumerate}
\end{corollary}
\begin{proof}
By Lemma \ref{lem:ESid} and Fact \ref{fact:E=EE}, $(S^\sharp,\ast)$ is isomorphic to $(E(S^\sharp),\circ)$ via the map taking $p$ in $S^\sharp$ to $q\mapsto p\ast q$ in $E(S^\sharp)$. So every function in $E(S^\sharp)$ is continuous by Lemma \ref{lem:ESid}$(ii)$. The claims now follow from Theorem \ref{thm:EN}. 
\end{proof}

\begin{remark}\label{rem:SWAP}
The previous arguments establish weak almost periodicity of $S^\sharp$ via work of Ellis and Nerurkar \cite{EllNer}. However, as previously mentioned, the general connection between stability and weak almost periodicity is well known in the model theory literature (e.g., \cite{BYGro,BYTs,HrKrP1}) and goes back to Grothendieck \cite{GroWAP}. Indeed, one can directly verify the definition of weak almost periodicity for $S^\sharp$ using \cite[Th\'{e}or\`{e}me 6]{GroWAP} and the small exercise that $\theta(x\cdot y)$ is stable for any $\varphi^\sharp$-formula $\theta(x)$. By going this route, one obtains a proof of part $(ii)$ of Lemma \ref{lem:ESid} using \cite{EllNer} and \cite{GroWAP} instead of Theorem \ref{thm:SFT}$(b)$.

   Note  that the same methods establish weak almost periodicity of $S$. Indeed, one can either check this directly using the definition and Grothendieck, or  use Theorem \ref{thm:SFT}$(b)$ in the same way to prove continuity of the maps $\sigma_p$ from the proof of Lemma \ref{lem:ESid}. More generally, if $f\colon S_1\to S_2$ is a surjective homomorphism of $G$-flows, and $S_1$ is weakly almost periodic, then so is $S_2$. In our case, the restriction map from $S^\sharp$ to $S$ is such a homomorphism. (In fact, any $G$-flow with a dense orbit admits a surjective homomorphism from its Ellis semigroup.)
 \end{remark}

For the rest of this section, we let $K$ denote the unique minimal subflow of $S^\sharp$.

\begin{definition}
Let $u$ be the identity in $K$, and define $\pi\colon G\to K$ by $\pi(g)=gu$. 
\end{definition}

\begin{proposition}\label{prop:pi-ext}
The map $\pi\colon G\to K$ is a $\varphi^\sharp$-definable compactification of $G$. Moreover, $p\mapsto p\ast u$ is the  unique continuous extension of $\pi$ to $S^\sharp$, and this map is a semigroup homomorphism.
\end{proposition}
\begin{proof}
By Corollary \ref{cor:EN}$(b)$, $p\mapsto p\ast u$ is a well-defined map from $S^\sharp$ to $K$, which clearly extends $\pi$ (viewed as a map on the dense set of realized types in $S^\sharp$). This map is also continuous (just because $S^\sharp$ is a right topological semigroup). So $\pi$ is $\varphi^\sharp$-definable. Note that $\pi(G)$ is the orbit of $u$, which is dense in $K$ by Corollary \ref{cor:EN}$(a)$. So to finish the proof, we just need to show that $p\mapsto p\ast u$ is a semigroup homomorphism. For this, fix $p,q\in S^\sharp$. Then
\[
p\ast q\ast u=p\ast q\ast u\ast u=p\ast u\ast q\ast u
\]
where the second equality uses Corollary \ref{cor:EN}$(d)$. 
\end{proof}

In light of the previous corollary, we will continue to use $\pi$ to denote the map $p\mapsto p\ast u$ from $S^\sharp$ to $K$. We also let $\mu$ denote the unique left-invariant Keisler $\varphi^\sharp$-measure on $G$ (which exists by Corollary \ref{cor:EN}$(e)$).

\begin{proposition}\label{prop:domination}
Let $\eta$ be the Haar functional on  $K$. Then $\mu(\theta(x))=\eta(\theta|_K)$ for any $\varphi^\sharp$-formula $\theta(x)$. In particular, $\mu(K)=1$.
\end{proposition}
\begin{proof}
The second claim follows from the first by regularity of $\mu$ and Urysohn's Lemma.
For the first claim, it suffices by uniqueness of $\mu$  to show that the map $\theta(x)\mapsto \eta(\theta|_K)$ is a left-invariant Keisler $\varphi^\sharp$-measure on $G$. It is clear that this map is a linear functional. Left-invariance  follows from left-invariance of $\eta$, and since $(g\theta)|_K=gu\ast \theta|_K$ for any $\varphi^\sharp$-formula $\theta(x)$ and any $g\in G$. 
\end{proof}

The next definition is a variation of  ``almost constant" (defined before Theorem \ref{thm:SARF}), which will be more manageable in later arguments with ultraproducts.

\begin{definition}\label{def:nearly}
Let $\theta(x)$ be a $\varphi^\sharp$-formula. Fix a Borel set $B\seq S_{\varphi^\sharp}(M)$ and some $r\in\R$. Then $\theta(x)$ is \textbf{$\mu$-nearly $\epsilon$-close to $r$ on $B$} if  $\mu(\{p\in B:\theta(p)\not\approx_\epsilon r\})=0$. 
\end{definition}

We can now prove the main result of this section.

\begin{theorem}\label{thm:mainUP}
Let $\theta(x)$ be a $\varphi^\sharp$-formula and fix some $\epsilon>0$. Then there is an open identity neighborhood $U\seq K$ such that for any $g\in G$, $\theta(x)$ is $\mu$-nearly $\epsilon$-close to $\theta(gu)$ on $g\pi\inv (U)$.
\end{theorem}
\begin{proof}
Recall that $\theta(x)$ is a continuous function on $S_{\varphi^\sharp}(M)$ and hence restricts to a continuous function the compact Hausdorff group $K$. So by Proposition \ref{prop:continuous} there is an open identity neighborhood $U\seq K$ such that $\theta(x)$ is $\epsilon$-constant on all left translates of $U$ in $K$. Now fix $g\in G$. We  show that $\theta(x)$ is $\mu$-nearly $\epsilon$-close to $\theta(gu)$ on $g\pi\inv (U)$. By Proposition \ref{prop:domination}, it suffices to fix $p\in g\pi\inv(U)$ with $\theta(p)\not\approx_\epsilon\theta(gu)$, and show $p\not\in K$. So fix such $p$. Then, using Proposition \ref{prop:pi-ext}, we have
\[
p\ast u=\pi(p)=\pi(g)\ast\pi(g\inv p)\in \pi(g)\ast U\mand gu=\pi(g)\ast u\in \pi(g)\ast U.
\]
Therefore $\theta(p\ast u)\approx_\epsilon\theta(gu)$ by choice of $U$, and so $\theta(p\ast u)\neq \theta(p)$ by choice of $p$. In particular, $p\ast u\neq p$, and hence $p\not\in K$ since $u$ is the identity in $K$.  
\end{proof}

Next we use the Peter-Weyl Theorem to specialize the previous result to unitary groups. For some ad hoc notation, let $B^n_\delta$ denote the open identity neighborhood of radius $\delta$ in $\U(n)$. 

\begin{corollary}\label{cor:mainUP}
Let $\theta(x)$ be a $\varphi^\sharp$-formula and fix some $\epsilon>0$. Then there are $n\geq 0$, $\delta>0$, and a $\varphi^\sharp$-definable homomorphism $\tau\colon G\to \U(n)$ such that for any $g\in G$, $\theta(x)$ is $\mu$-nearly $\epsilon$-close to $\theta(gu)$ on $g\tau\inv (B^n_\delta)$.\footnote{As with the map $\pi$ in Theorem \ref{thm:mainUP}, here we identify $\tau$ with its unique continuous extension to $S_{\varphi^\sharp}(M)$ so that $\tau\inv(B^n_\delta)$ make sense in the context of Definition \ref{def:nearly}.}
\end{corollary}
\begin{proof}
Let $U\seq K$ be an open identity neighborhood in $K$ as in Theorem \ref{thm:mainUP}. Recall that $K$ is a compact Hausdorff group. By the Peter-Weyl Theorem, $K$ is an inverse limit $\varprojlim_I K_i$ of compact Lie groups (see  \cite[Corollary 2.43]{HofMo3}). For $i\in I$, let $\rho_i\colon K\to K_i$ be the projection map. Then there is some $i\in I$ and an open identity neighborhood $V\seq K_i$ such that $\rho_i\inv(V)\seq U$ (see \cite[Exercise 1.1.15]{RZbook}). By Peter-Weyl again, we may fix some $n\geq 0$ and a topological embedding $\iota\colon K_i\to \U(n)$ of $K_i$ as a closed subgroup of $\U(n)$ (see \cite[Theorem 6.1.2]{Kowalski-book}). Pick $\delta>0$ such that $\iota\inv(B^n_{\delta})\seq V$. Then $\tau\coloneqq \iota\rho_i\pi\colon G\to \U(n)$ is a $\varphi^\sharp$-definable homomorphism and $\tau\inv(B^n_\delta)\seq\pi\inv(U)$. The conclusion now follows from the choice of $U$.
\end{proof}

\begin{remark}\label{rem:inmodel}
In  analogy to \cite{CoLSGT}, all of the material in this section can be generalized to the following  abstract setting. Let $G$ be a group (without any explicit model-theoretic context). Suppose $\cF$ is a collection of real-valued bounded stable functions on $G$, which is closed under left-translation. Then we have the associated  ``type space" $S(\cF)$ (i.e., the space of functions $p\colon \cF\to \R$ that are finitely approximated in $G$) and notion of an ``$\cF$-formula" (i.e., a continuous function on $S(\cF)$). Moreover, $S(\cF)$ is a $G$-flow by left-invariance of $\cF$. The same arguments show that $S(\cF)$ is weakly almost periodic with Ellis semigroup $S(\cF^\sharp)$, where $\cF^\sharp$ is the set of functions of the form $f(xg)$ for $f(x)\in \cF$ and $g\in G$. Moreover, Theorem \ref{thm:mainUP}  holds for any $\cF^\sharp$-formula. In fact, since Theorem \ref{thm:SFT} holds under the weaker notion of ``stability in a model" (see, e.g.,  \cite{BYGro} or \cite[Appendix B]{BYU}), it suffices to just assume that each $f\in \cF$ is \emph{stable in $G$}, i.e., for all $\epsilon>0$ there do not exist infinite sequences $(a_i)_{i<\omega}$ and $(b_i)_{i<\omega}$ in $G$ such that $|f(a_ib_j)-f(a_jb_i)|\geq\epsilon$ for all $i<j<\omega$. 
\end{remark}

\section{The main result and corollaries}\label{sec:main}

We now have all of the tools necessary to prove Theorem \ref{thm:SARF} (the main result of the paper), which will combine Corollary \ref{cor:mainUP} with an ultraproduct construction. As was the case in \cite{CCP}, this process will be more difficult than analogous constructions in classical logic. Indeed, we will need some of the  formalism from \cite{CCP} for dealing with  discrepancies between ultralimits of Keisler measures as functionals versus as  probability measures. We will also need to deal with the fact that definable compactifications do not necessarily yield Bohr neighborhoods given  explicitly by formulas. This issue previously arose  in \cite{CPTNIP} (with Terry) where, in the setting of classical logic, we used definable ``approximate Bohr neighborhoods" that were later replaced by genuine Bohr neighborhoods using a modification of a result of Kazhdan \cite{Kazh} due to Alekseev, Glebskii, and Gordon \cite{AGG}. Here we will avoid this extra level of approximation by employing  machinery from \cite{CHP}, which uses Kazhdan's result at an earlier stage of the process, and builds definable homomorphisms to a compact Lie group into the underlying continuous logic. In fact, these tools from \cite{CHP} were  developed with the present application in mind.

\begin{proof}[\textnormal{\textbf{Proof of Theorem \ref{thm:SARF}.}}]
Suppose the result fails for the fixed parameters $\epsilon$, $\zeta$, and $\bk$ from the  statement of the theorem. Then for every $s\geq 1$ there is an amenable group $G_s$, with a left-invariant measure $\mu_s$, and a $\bk$-stable function $f_s\colon G_s\to [\nv 1,1]$, such that for any $(\delta,\U(n))$-Bohr neighborhood $B$ in $G_s$, with $\delta\inv, n\leq s$, the function $f_s$ is not $\zeta(\delta,n)$-almost $\epsilon$-constant on all translates $B$. 

Let $\cL^0$ be a continuous first-order language containing the language of groups together with a $[\nv 1,1]$-valued unary relation symbol $f$, all with trivial moduli of uniform continuity. We equip $G_s$ with the discrete metric, and  expand to an $\cL^0$-structure by interpreting the language of groups in the obvious way and interpreting $f$ as $f_s$.  Fix a nonprincipal ultrafilter $\cU$ on $\Z^+$ and let $G$ be the  ultraproduct $\prod_{\cU}G_s$. So $G$ is an $\cL^0$-structure expanding a group. Let $\varphi(x,y)$ be the $\cL^0$-formula $f(y\cdot x)$. By assumption, $\varphi(x,y)$ is $\bk$-stable in $\Th(G_s)$ for all $s\geq 1$, and thus $\varphi(x,y)$ is stable in $\Th(G)$ by Proposition \ref{prop:stableup}. Let $\mu_\varphi$ be the unique left-invariant Keisler $\varphi^\sharp$-measure on $G$, which exists by Corollary \ref{cor:EN}$(e)$.
\medskip

\noindent\textit{Claim 1.} There are $n\geq 0$, $\delta>0$, and a $\varphi^\sharp$-definable homomorphism $\tau\colon G\to \U(n)$ such that for any $g\in G$, there is some $h\in G$ such that $f(x)$ is $\mu_\varphi$-nearly $\frac{\epsilon}{4}$-close to $f(h)$ on $g\tau\inv(B^\delta_n)$.\footnote{Recall that $B^\delta_n$ denotes the open identity neighborhood of radius $\delta$ in $\T(n)$.}

\noindent\emph{Proof.} Corollary \ref{cor:mainUP} yields this statement but with $h$ a $\varphi^\sharp$-type (namely $gu$). We can fix this by applying Corollary \ref{cor:mainUP} with $\frac{\epsilon}{8}$, and then for each $g\in G$ finding $h\in G$ with $f(h)\approx_{\epsilon/8}f(gu)$.\clqed\medskip

Following \cite[Section 4.3]{CHP}, we  view $(G,\U(n),\tau)$ as a structure in a two-sorted language $\cL$ expanding the  $\cL^0$-structure on $G$ (in \cite{CHP}, $\cL$ would be denoted  $(\cL^0)^{\U(n)}_{\tau}$).    By \cite[Corollary 4.12]{CHP}, for each $s\geq 1$, there is a homomorphism $\tau_s\colon G_s\to \U(n)$ such that the $\cL$-structure $(G,\U(n),\tau)$ is canonically isomorphic to $\prod_{\cU}(G_s,\U(n),\tau_s)$.  Viewing each $\mu_s$ as a positive linear functional on $G_s$, define the ultralimit functional $\mu=\lim_{\cU}\mu_s$ (see \cite[Proposition 2.15]{CCP}). Then $\mu$ is a left-invariant Keisler measure on the $\cL$-structure $(G,\U(n),\tau)$ in the sort for $G$.

Recall the continuous connective $r\dotminus s=\max\{r-s,0\}$. Define the $\cL$-formula 
\[
\psi(x,y,z)=\min\{\delta\dotminus d(\tau(x),\tau(y)),|f(x)-f(z)|\dotminus{\textstyle\frac{\epsilon}{4}}\}.
\]

\noindent\emph{Claim 2.} For any $g\in G$ there is some $h\in G$ such that $\mu(\psi(x,g,h)>0)=0$.

\noindent\emph{Proof.} By Claim 1, it suffices to fix $g,h\in G$ and show 
\[
\textstyle \mu(\psi(x,g,h)>0)=\mu_\varphi(\{p\in g\tau\inv(B^\delta_n):|f(p)-f(h)|>\frac{\epsilon}{4}\}).
\]
Let $\Phi$ consist of the $\cL$-formulas $\varphi^\sharp(x;y,z)$ and $\chi(x;y)=d(\tau(x),\tau(y))$. So any instance of $\psi(x;y,z)$ is a $\Phi$-formula. Since $\tau$ is $\varphi^\sharp$-definable, every type in $S_{\varphi^\sharp}(G)$ has a unique extension to a type in $S_\Phi(G)$.  So the restriction of $\mu$ to $\Phi$-formulas can be identified with $\mu_\varphi$, which gives us the desired result. \clqed \medskip

In order to transfer Claim 2 through the ultraproduct, we will further expand the language as in  \cite[Section 5]{CCP}. Define $\cL^+$ to consist of $\cL$ together with, for each $\eta>0$, a new $[0,1]$-valued predicate symbol $Q_\eta(y,z)$ with a trivial modulus of uniform continuity. Expand $(G_s,\U(n),\tau_s)$ to an $\cL^+$-structure $G^+_s$ by interpreting $Q_\eta$ as the function $(g,h)\mapsto \mu_s(\psi(x,g,h)\geq \eta)$. Let $G^+=\prod_{\cU}G^+_s$, and note that the reduct of $G^+$ to $\cL$ is $(G,\U(n),\tau)$. To ease notation, we use $Q_\eta$ and $Q^s_\eta$ for the interpretations of $Q_\eta$ in $G^+$ and $G^+_s$, respectively.
\medskip

\noindent\emph{Claim 3.} For any $\eta>0$ and $g,h\in G$, $Q_\eta(g,h)\leq \mu(\psi(x,g,h)\geq\eta)$. 

\noindent\emph{Proof.} See \cite[Lemma 5.4]{CCP}, which is written for more general families of formulas, but in the context of ultraproducts of finite structures with the normalized counting measure. The proof works verbatim in our setting, and  involves expanding by additional  predicate symbols (see the start of \cite[Section 5.1]{CCP}; in our setting the family of $\cL$-formulas $\cF$ can be taken as the closure of $\{\psi(x,y,z)\}$ under the  connectives  $\alpha_D$ mentioned there). So here we are again using the fact that ultraproducts commute with reducts (or one can just work in a larger language).\clqed\medskip

Now let $c$ be the absolute constant from Proposition \ref{prop:Bohr}$(a)$, and define
\[
\textstyle \eta=\min\left\{\frac{\delta}{2},\frac{\epsilon}{4},\zeta(\frac{\delta}{2},n)\right\}.
\]
Combining Claims 2 and 3, we have 
$\sup_x\inf_y Q_{\eta}(x,y)=0$. Let $I\seq\Z^+$ be the set of $s\geq 1$ such that $\sup_x\inf_y Q^s_{\eta}(x,y)<\eta$. Then $I\in\cU$ and so, since $\cU$ is nonprincipal, we may fix some $s\in I$ such that $n,(\delta/2)\inv\leq s$. Let $B=\tau_s\inv(B^{\delta/2}_n)$, and note that $B$ is a $(\frac{\delta}{2},\U(n))$-Bohr neighborhood in $G_s$. To finish the proof, we will show that $f_s$ is $\zeta(\frac{\delta}{2},n)$-almost $\epsilon$-constant on all translates of $B$, contradicting the initial choice of $G_s$ and $f_s$. 

Fix $g\in G_s$. We need to show that $f_s$ is $\zeta(\frac{\delta}{2},n)$-almost $\epsilon$-constant on $gB$. Since $\sup_x\inf_y Q^s_\eta(x,y)<\eta$, there is some $h\in G_s$ such that $\mu_s(\psi(x,g,h)\geq \eta)<\eta$.
Let $Z$ denote the set in $G_s$ defined by $\psi(x,g,h)\geq \eta$. So $\mu_s(Z)<\eta\leq \zeta(\frac{\delta}{2},n)$.
Thus it suffices to show that $f_s$ is $\epsilon$-constant on $gB\backslash Z$. 

From the definition of $\psi(x,y,z)$, we have
\[
Z= \{x\in G_s:d(\tau_s(x),\tau_s(g))\leq\delta-\eta\text{ and }|f_s(x)-f_s(h)|\geq {\textstyle\frac{\epsilon}{4}}+\eta\}.
\] 
Note also that if $x\in gB$ then $d(\tau_s(x),\tau_s(g))<\frac{\delta}{2}\leq\delta-\eta$, and hence if $x\in gB\backslash Z$ then we must have $|f_s(x)-f_s(h)|<\frac{\epsilon}{4}+\eta\leq\frac{\epsilon}{2}$. So $f_s$ is $\epsilon$-constant on $gB\backslash Z$ by the triangle inequality, as desired.   
\end{proof}

Next we   show that when restricting Theorem \ref{thm:SARF} to certain subclasses of amenable groups, one can replace unitary Bohr neighborhoods  by ``nicer" objects. We start by recalling that any abelian group is amenable (see, e.g., \cite[Proposition 0.15]{Patbook}). So Theorem \ref{thm:SARF} and Proposition \ref{prop:Bohr}$(b)$  yield the following conclusion. 

\begin{corollary}\label{cor:abelian}
Fix  $\epsilon>0$, $\zeta\colon\R^+\times \N\to\R^+$, and $\bk\colon \R^+\to\Z^+$. Suppose $G$ is an abelian group, with left-invariant measure $\mu$, and $f\colon G\to[\nv 1,1]$ is $\bk$-stable. Then there is a $(\delta,\T(n))$-Bohr neighborhood $B$ in $G$, with $\delta\inv,n\leq O_{\bk,\zeta,\epsilon}(1)$, such that $f$ is $\zeta(\delta,n)$-almost $\epsilon$-constant on all translates of $B$. 
\end{corollary}

Next we consider  amenable groups of bounded exponent.

\begin{corollary}\label{cor:exponent}
Fix  $r\geq 1$, $\epsilon>0$, $\zeta\colon\Z^+\to\R^+$, and  $\bk\colon \R^+\to\Z^+$. Suppose $G$ is an amenable group of exponent $r$, with left-invariant measure $\mu$, and $f\colon G\to [\nv 1,1]$ is $\bk$-stable. Then there is a normal subgroup $H\leq G$ of index $m\leq O_{\bk,\zeta,\epsilon,r}(1)$ such that $f$ is $\zeta(m)$-almost $\epsilon$-constant on all cosets of $H$. 
\end{corollary}
\begin{proof}
Let $c$ be the absolute constant from Proposition \ref{prop:Bohr}$(a)$, and let $\gamma_r>0$ be the $O_r(1)$ parameter from Proposition \ref{prop:Bohr}$(c)$. 
Define $\zeta^*\colon\R^+\times\N\to\R^+$ so that
\[
\zeta^*(\delta,n)=\min\{\zeta(m):m\leq (c/\min\{\delta,\gamma_r\})^{n^2}\}.
\]
Now fix $G$, $\mu$, and $f$ as in the statement of the corollary.
Apply Theorem \ref{thm:SARF} with parameters $\bk$, $\zeta^*$, and $\epsilon$  to obtain a $(\delta,\U(n))$-Bohr neighborhood $B$ in $G$, with $\delta\inv,n\leq O_{\bk,\zeta,\epsilon,r}(1)$, such that $f$ is $\zeta^*(\delta,n)$-almost $\epsilon$-constant on all translates of $B$. Let $\delta_*=\min\{\delta,\gamma_r\}$ and note that we still have $\delta_*\inv\leq O_{\bk,\zeta,\epsilon,r}(1)$. Let $H$ be the $(\delta_*,\U(n))$-Bohr neighborhood in $G$ defined using the same homomorphism to $\U(n)$ that yields $B$. Then $H$ is a normal subgroup of $G$ by Proposition \ref{prop:Bohr}$(c)$, and $H\seq B$. Moreover, $H$ has index $m\leq (c/\delta_*)^{n^2}$ in $G$ by Proposition \ref{prop:Bohr}$(a)$. So $\zeta^*(\delta,n)\leq \zeta(m)$.

Now fix $g\in G$. Then there is a set $Z\seq gB$, with $\mu(Z)<\zeta^*(\delta,n)$ such that $f$ is $\epsilon$-constant on $gB\backslash Z$. Set $W=Z\cap gH$. Then $gH\backslash W\seq gB\backslash Z$, so $f$ is $\epsilon$-constant on $gH\backslash W$. Moreover, $\mu(W)\leq\mu(Z)<\zeta^*(\delta,n)\leq\zeta(m)$. Therefore $f$ is $\zeta(m)$-almost $\epsilon$-constant on all cosets of $H$. 
\end{proof}

Finally,  we consider  torsion groups. In this setting, it turns out that  unitary Bohr neighborhoods are closely connected to Bohr neighborhoods defined from a torus. This phenomenon plays an implicit role in \cite{CPTNIP} and  \cite{CoBogo} where, for a finite group $G$, the key structural ingredient is a  $\T(n)$-Bohr neighborhood in a bounded index subgroup of $G$. However, the ambient unitary representation was not evident in the setting of those papers. This motivates the following definition.

\begin{definition}\label{def:Tmn}
Fix a group $G$, a real number $0<\delta\leq 2$, and integers $n\geq 0$ and $m\geq 1$. Then a \textbf{$(\delta,n,m)$-Bohr neighborhood in $G$} is a set of the form $\tau\inv(U\cap K)$, where $\tau\colon G\to \U(n)$ is a group homomorphism, $U$ is the open identity neighborhood in $\U(n)$ of radius $\delta$, and $K$ is a normal subgroup of $\tau(G)$, which is contained in $\T(n)$ and has index $m$ in $\tau(G)$. 
\end{definition}

The next result  is a direct corollary of the Jordan-Schur Theorem for torsion subgroups of $\GL_n(\C)$. It was previously stated in  \cite[Proposition 5.4$(b)$]{CHP} in a  superficially weaker form, but follows using the same proof. 

\begin{proposition}\label{prop:Bohr-torsion}
Let $B$ be a $(\delta,\U(n))$-Bohr neighborhood in a torsion group $G$. Then $B$ contains a $(\delta,n,m)$-Bohr neighborhood in $G$ for some $m\leq O_n(1)$. 
\end{proposition}

\begin{remark}\label{rem:nmBohr}
Let $G$ be a group, and suppose $B$ is a $(\delta,n,m)$-Bohr neighborhood in $G$, witnessed by $\tau$ and $K$.
\begin{enumerate}[$(1)$]
\item  $B$ is a $(\delta,\T(n))$-Bohr neighborhood in the normal subgroup $\tau\inv(K)$ of $G$, which has index at most $m$. So we recover the setting of \cite{CPTNIP,CoBogo} discussed above. 
\item $B$ is  ``normal" in the sense that $gB=Bg$ for any $g\in G$.\footnote{In \cite{CPTNIP}, this was erroneously claimed to hold for the weaker notion of $\T(n)$-Bohr neighborhood in a subgroup; see the supplemental note to \cite{CPTNIP} on the first author's  webpage.} 
\item Proposition \ref{prop:Bohr}$(a)$ implies that $G$ can be covered by at most $m(c/\delta)^{n^2}$ translates of $B$. But this can be improved in light of point $(1)$ above. Indeed, let $U$  be the the open identity neighborhood of radius $\delta/2$ in $\T(1)$. By \cite[Corollary 5.6]{CHP}, $G$ can be covered by $m\ell^n$ translates of $B$ where $\ell\geq 1$ is any integer with the property that $\T(1)$ can be covered by $\ell$ translates of $U$. Since $U$ is a cyclic interval of  arclength $4\sin\inv(\delta/4)>\delta$, we may take $\ell=\lceil 2\pi/\delta\rceil$.  
\end{enumerate}
\end{remark}

\begin{corollary}\label{cor:torsion}
Fix  $\epsilon>0$, $\zeta\colon\R^+\times \N\times\Z^+\to\R^+$, and $\bk\colon \R^+\to\Z^+$. Suppose $G$ is an amenable torsion group, with left-invariant measure $\mu$, and $f\colon G\to [\nv 1,1]$ is $\bk$-stable. Then there is a $(\delta,n,m)$-Bohr neighborhood $B$ in $G$, with $\delta\inv,n,m\leq O_{\bk,\zeta,\epsilon}(1)$, such that $f$ is $\zeta(\delta,n,m)$-almost $\epsilon$-constant on all translates of $B$. 
\end{corollary}
\begin{proof}
Let $m\colon\N\to \N$ be the  function given by $O_n(1)$ in Proposition \ref{prop:Bohr-torsion}.
Define $\zeta^*\colon \R^+\times\N\to \R^+$ by  $\zeta^*(\delta,n)=\min\{\zeta(\delta,n,m):m\leq m(n)\}$.

Now fix $G$, $\mu$, and $f$ as in the statement of the corollary. Apply Theorem \ref{thm:SARF} with parameters $\bk$, $\zeta^*$, and $\epsilon$ to obtain a $(\delta,\U(n))$-Bohr neighborhood $B_0$ in $G$, with $\delta\inv,n\leq O_{\bk,\zeta,\epsilon}(1)$, such that $f$ is $\zeta^*(\delta,n)$-almost $\epsilon$-constant on all translates of $B_0$. By Proposition \ref{prop:Bohr-torsion}, there is a $(\delta,n,m)$-Bohr neighborhood $B\seq B_0$ with $m\leq m(n)$. So $\zeta^*(\delta,n)\leq \zeta(\delta,n,m)$. Arguing as in the end of the proof of Corollary \ref{cor:exponent}, it follows that $f$ is $\zeta(\delta,n,m)$-almost $\epsilon$-constant on all translates of $B$.
\end{proof}

\section{Applications}\label{sec:app}

In this section, we use our main  result (Theorem \ref{thm:SARF}) and  stability of convolutions (Corollary \ref{cor:conv-stable}) to prove several results in arithmetic combinatorics related to \emph{Bogolyubov's Lemma}  (discussed in Sections \ref{sec:BL} and \ref{sec:AB}), \emph{quasirandom groups} (discussed in Section \ref{sec:AB}), and the \emph{Croot-Sisask Lemma} (discussed in Section \ref{sec:CS}). Beyond  proving applications, the general aim of this section is to illustrate the fundamental role played by stable arithmetic regularity for  functions in the context of a significant body of  literature. In particular, we will draw a  thread  through the contemporaneous papers of Hrushovski \cite{HruAG}, Sanders \cite{SanBS}, and Croot and Sisask \cite{CrSi}, which together laid the foundation for the seminal results of Breuillard, Green, and Tao \cite{BGT} on the structure of approximate groups. As is now well understood, certain aspects Hrushovski's work  \cite{HruAG} can be accounted for by stability of Hilbert spaces in continuous logic (Fact \ref{fact:Hilbert}), which provides a link to the methods used here. Moreover, work of Palac\'{i}n \cite{PalPFS} establishes a direct connection between \cite{HruAG}  and results of Gowers \cite{GowQRG} on quasirandom groups.

\subsection{Bogolyubov's Lemma}\label{sec:BL}
In this subsection we use Corollary \ref{cor:conv-regularity} to give a new proof of Bogolyubov's Lemma for amenable groups (see Theorem \ref{thm:strongBL}). This result was first proved for finite abelian groups  by Ruzsa \cite{Ruz94} (with explicit bounds) using an argument based on much earlier techniques developed by Bogolyubov \cite{Bog39}. A qualitative generalization to arbitrary finite groups was later given by the first author in \cite{CoBogo} using a combinatorial argument of Sanders \cite{SanBS}. Very recently, the authors and Hrushovski \cite{CHP} proved a further generalization to arbitrary amenable groups using a variation of Hrushovski's stabilizer theorem \cite{HruAG} due to Montenegro, Onshuus, and Simon \cite{MOS}. Here we prove the version for amenable groups directly from Corollary \ref{cor:conv-regularity} and elementary covering arguments. 

\begin{theorem}\label{thm:strongBL}
Let $G$ be an amenable group with left-invariant measure $\mu$, and fix $\alpha>0$. Suppose $A\seq G$ is such that $\mu(A)\geq\alpha$. Then there is a $(\delta,\U(n))$-Bohr neighborhood $B\seq G$, with $\delta\inv,n\leq O_{\alpha}(1)$, such that $B\seq (AA\inv)^2$. 
\end{theorem}
\begin{proof}
Define $f\colon G\to [0,1]$ such that $f(x)=\mu(A\cap xA)$. Then one can check $f=\boldsymbol{1}_A\ast\boldsymbol{1}_{A\inv}$. The first step is to show that $f$ is bounded away from zero on a large set. In particular, set $\epsilon=\frac{1}{2}\alpha^2$ and let $S=\{x\in G:f(x)>\epsilon\}$. \medskip

\noindent\textit{Claim 1.} $G$ can be covered by at most $2\alpha\inv$ left translates of $S$. 

\noindent\textit{Proof.} Call a set $F\seq G$ \emph{separated} if $\mu(gA\cap hA)\leq \epsilon$ for all distinct $g,h\in F$. We first show that any separated set $F$ has size at most $2\alpha\inv$. For a contradiction, suppose $F\seq G$ is separated with $|F|>2\alpha\inv$. Then we may pick pairwise distinct $g_1,\ldots,g_n\in F$, with $n=\lceil\frac{2}{\alpha}\rceil$. So $\frac{2}{\alpha}\leq n<\frac{2+\alpha}{\alpha}$. Moreover,
\[
1\geq \mu\left(\bigcup_{i=1}^n g_iA\right) \geq \sum_{i=1}^n \mu(g_iA)-\sum_{1\leq i<j\leq n}\mu(g_iA\cap g_jA)\geq n\alpha-\frac{n(n-1)\epsilon}{2}.
\]
 $1\geq n\alpha\beta$ where $\beta=1-\frac{1}{4}(n-1)\alpha$ (recall $\epsilon=\frac{1}{2}\alpha^2$). Since $n\alpha\geq 2$, we must  have $2\beta\leq 1$. But one can directly check that this contradicts $n\alpha<2+\alpha$.

Now let $F\seq G$ be a separated set of maximal size. We show $G=FS$. Choose $g\in G$. Then by maximality there is some $h\in F$ such that $\mu(gA\cap hA)>\epsilon$, i.e., $\mu(h\inv gA\cap A)>\epsilon$, i.e., $h\inv g\in S$, i.e., $g\in hS$.\clqed\medskip

Now define $\zeta\colon\R^+\times\N\to \R^+$ so that
\[
\zeta(\delta,n)=\textstyle\min\{\frac{\alpha}{2}(\frac{\delta}{c})^{n^2},\frac{1}{2}(\frac{\delta}{2c})^{n^2}\},
\]
where $c$ is the absolute constant from Proposition \ref{prop:Bohr}$(a)$.
Apply Corollary \ref{cor:conv-regularity} to $f$ with parameters $\epsilon$ and $\zeta$. This yields a $(\delta,\U(n))$-Bohr neighborhood $C$ in $G$, with $\delta\inv, n\leq O_{\alpha}(1)$, such that $f$ is $\zeta(\delta,n)$-almost $\epsilon$-constant on all translates of $C$. 
\medskip

\noindent\textit{Claim 2.} There is some $g\in G$ such that $\mu(gC\backslash AA\inv)< \zeta(\delta,n)$.

\noindent\emph{Proof.} By Proposition \ref{prop:Bohr}$(a)$, there is a set $F\seq G$ of size at most $(c/\delta)^{n^2}$ such that $G=FC$. For each $g\in F$, fix $Z_g\seq gC$ such that $\mu(Z_g)<\zeta(\delta,n)$ and $f$ is $\epsilon$-constant on $gC\backslash Z_g$. Let $Z=\bigcup_{g\in F}Z_g$. Then $\mu(Z)<\zeta(\delta,n)|F|\leq\frac{\alpha}{2}$. Since $\mu(S)\geq\frac{\alpha}{2}$ by Claim 2, we may fix some $x\in S\backslash Z$. Choose $g\in F$ such that $x\in gC$, and note $x\not\in Z_g$. So given any $y\in gC\backslash Z_g$, we have $f(x)\approx_\epsilon f(y)$, and hence $f(y)>0$ since $f(x)>\epsilon$. In particular, $A\cap yA$ is nonempty for all $y\in gC\backslash Z_g$, which implies $gC\backslash Z_g\seq AA\inv$, i.e., $gC\backslash AA\inv\seq Z_g$. This yields the claim by choice of $Z_g$.  \clqed\medskip

Fix $g\in G$ as in Claim 2, and let $B$ be the $(\frac{\delta}{2},\U(n))$-Bohr neighborhood in $G$ defined from the same homomorphism to $\U(n)$ that yields $C$. We will show $B\seq (AA\inv)^2$ using the  following standard covering exercise (the proof is an elementary application of the inclusion-exclusion principle and is left to the reader).\medskip

\noindent\emph{Claim 3.} Given $X,Y\seq G$, if $1\in X$ and $\mu(X^2\backslash Y)<\frac{1}{2}\mu(X)$, then $X\seq YY\inv$. \medskip

Note that $B^2\seq C$, hence $\mu(B^2\backslash g\inv AA\inv)<\zeta(\delta,n)$ by Claim 2. Also, $\zeta(\delta,n)\leq \frac{1}{2}\mu(B)$ by Proposition \ref{prop:Bohr}$(a)$ and  choice of $\zeta$. Applying Claim 3 with $X=B$ and $Y=g\inv AA\inv$, we have $B\seq g\inv (AA\inv)^2 g$. So $B\seq (AA\inv)^2$ since $gB=Bg$. 
\end{proof}

\begin{remark}
Claim 1 is related to \emph{Ruzsa's Covering Lemma} which, in the setting of Theorem \ref{thm:strongBL}, says that if $A\seq G$ and $\mu(A)\geq\alpha>0$ then $G$ can be covered by at most $\alpha\inv$ left translates of $AA\inv$. The set $S$ in the proof of Theorem \ref{thm:strongBL} is contained in $AA\inv$, since the latter is precisely the set of $x\in G$ such that $A\cap xA\neq\emptyset$. Thus Claim 1 can be viewed as a ``density analogue" of Ruzsa's Covering Lemma. 
\end{remark}

\begin{remark}\label{rem:BL}
Using a suitable adjustment to $\zeta$ and another covering exercise similar to Claim 3, one can further show that $AAA\inv$ contains a translate of $B$ (see also Claim 3 in the proof of Theorem \ref{thm:AB} below). This recovers  the strong form of Bogolyubov's Lemma in  \cite[Theorem 5.9]{CHP}, except with $AA\inv$ almost containing a translate of a Bohr neighborhood, rather than the Bohr neighborhood itself.  See also \cite[Remark 5.11]{CHP}. 
\end{remark}

\subsection{Two-set Bogolyubov and quasirandom groups}\label{sec:AB}

In this subsection we prove a  ``two-set" version of Bogolyubov's Lemma for finite groups, which is a new result in the noncommutative setting (some discussion of related results for abelian groups is given after the proof). In fact, we will prove a more detailed statement that matches the strong form of Bogolyubov (for one set) from \cite{CHP}. Specifically, our result provides an ``almost containment" statement in a double product (condition $(i)$ of Theorem \ref{thm:AB} below).  As was the case in \cite{CHP}, from this one can deduce the traditional form of Bogolyubov in terms of full containment in a quadruple product (condition $(iii)$ below), as well as another containment statement for a triple product (condition $(ii)$). These three conditions were formulated in \cite{CHP} (for one set)  to match the form of Hrushovski's original stabilizer theorem from \cite{HruAG}. The generalization to two sets is also motivated by  results in combinatorial number theory along the lines of Jin's Theorem \cite{JinAB} (see  the discussion at the end of this subsection). Condition $(i)$ also immediately implies a  result of Gowers \cite{GowQRG} on quasirandom groups, but without quantitative bounds (see Corollary \ref{cor:Gow}).

\begin{theorem}\label{thm:AB}
Fix $\alpha>0$ and $\zeta\colon \R^+\times \N\to \R^+$. Let $G$ be a finite group and suppose $A,B\seq G$ are such that $|A|\geq \alpha|G|$ and $|B|\geq\alpha|G|$. Then there is a $(\delta,\U(n))$-Bohr neighborhood $U$ in $G$, with $\delta\inv, n\leq O_{\alpha,\zeta}(1)$, such that:
\begin{enumerate}[$(i)$]
\item $|gU\backslash AB|<\zeta(\delta,n)|G|$ for some $g\in G$,
\item $ABA\inv$ contains a translate of $U$, and
\item $U\seq AB(AB)\inv$.
\end{enumerate}
\end{theorem}
\begin{proof}
The  strategy is similar to the proof of Theorem \ref{thm:strongBL}. Define $f=\boldsymbol{1}_A\ast \boldsymbol{1}_B$ (i.e, $f(x)=|A\cap xB\inv|/|G|$). We first show that $f$ is bounded away from zero on a large set. In particular, set $\epsilon=\frac{1}{2}\alpha^2$ and let $S=\{x\in G:f(x)>\epsilon\}$.\medskip

\noindent\textit{Claim 1.} $|S|\geq\epsilon|G|$. 

\noindent\textit{Proof.} The key observation is $\sum_{x\in G}f(x)=|A||B|/|G|$, which follows directly from the definition of $f$ and Fubini. Therefore we have
\[
2\epsilon|G|=\alpha^2|G|\leq \frac{|A||B|}{|G|}=\sum_{x\in G}f(x)=\sum_{x\in S}f(x)+\sum_{x\in G\backslash S}f(x)\leq |S|+\epsilon|G\backslash S|.
\]
So $|S|\geq 2\epsilon|G|-\epsilon|G\backslash S|\geq \epsilon|G|$. \clqed\medskip

Now define $\zeta^*\colon \R^+\times\Z^+\to\R^+$ so that
\[
\zeta^*(\delta,n)=\min\textstyle\left\{\zeta(\delta,n),\epsilon\left(\frac{\delta}{c}\right)^{n^2},\min\left\{\alpha,\frac{1}{2}\right\}\left(\frac{\delta}{2c}\right)^{n^2}\right\},
\]
where $c$ is the absolute constant from Proposition \ref{prop:Bohr}$(a)$. Apply Corollary \ref{cor:conv-regularity} to $f$ with parameters $\epsilon$ and $\zeta^*$, and with $\mu$ the normalized counting measure on $G$. This yields a $(\delta,\U(n))$-Bohr neighborhood $V$ in $G$, with $\delta\inv,n\leq O_{\alpha,\zeta}(1)$, such that $f$ is $\zeta^*(\delta,n)$-almost $\epsilon$-constant on all  translates of $V$.\medskip 

\noindent\textit{Claim 2.} There is some $g\in G$ such that $|gV\backslash AB|<\zeta^*(\delta,n)|G|$.

\noindent\textit{Proof.} The argument is nearly identical to the proof of Claim 2 in Theorem \ref{thm:strongBL}. By Proposition \ref{prop:Bohr}$(a)$, there is a set $F\seq G$ of size at most $(c/\delta)^n$ such that $G=FV$. For each $g\in F$, fix $Z_g\seq gV$ such that $|Z_g|<\zeta^*(\delta,n)|G|$ and $f$ is $\epsilon$-constant on $gV\backslash Z_g$. Let $Z=\bigcup_{g\in F}Z_g$. Then $|Z|<\zeta^*(\delta,n)|F||G|\leq \epsilon|G|$. By Claim 1, there is some $x\in S\backslash Z$. Choose $g\in F$ such that $x\in gV$, and note  $x\not\in Z_g$. It then follows that $f(y)>0$ for all $y\in gV\backslash Z_g$. Since $f(y)=|A\cap yB\inv|/|G|$, we have $A\cap yB\inv\neq\emptyset$ for all $y\in gV\backslash Z_g$, i.e., $gV\backslash Z_g\seq AB$. So $|gV\backslash AB|\leq |Z_g|<\zeta^*(\delta,n)|G|$. \clqed\medskip

Fix $g\in G$ as in Claim 2, and let $U$ be the $(\frac{\delta}{2},\U(n))$-Bohr neighborhood in $G$ defined from the same homomorphism to $\U(n)$ that yields $V$.
We show that $U$ satisfies conditions $(i)$, $(ii)$, and $(iii)$. 

For $(i)$, we have $|gU\backslash AB|\leq |gV\backslash AB|<\zeta^*(\delta,n)|G|\leq  \zeta(\delta,n)|G|$.

For $(iii)$,  note that $U^2\seq V$ and, by choice of $\zeta^*$ and Proposition \ref{prop:Bohr}$(a)$, we have $|V\backslash g\inv AB|<\zeta^*(\delta,n)|G|\leq\frac{1}{2}(\delta/2c)^{n^2}|G|\leq \frac{1}{2}|U|$.
So  $U\seq g\inv AB(AB)\inv g$ by Claim 3 in the proof of Theorem \ref{thm:strongBL}. Hence $U=gUg\inv\seq AB(AB)\inv$.

For $(ii)$, we appeal to another  covering exercise similar to Claim 3 in Theorem \ref{thm:strongBL} (and alluded to in Remark \ref{rem:BL}). The proof is again left to the reader.
\medskip

\noindent\textit{Claim 3.} Fix $X\seq G$ such that $X=X\inv$ and $G$ can be covered by at most $k$ left translates of $X$. Then for any $C,D\seq G$, if $|X^2\backslash D|<\frac{1}{k}|C|$ then $CD\inv$ contains a left translate of $X$.\medskip

We apply this with $C=A$, $X=U$, $D=g\inv AB$, and $k=(2c/\delta)^{n^2}$. Recall $U^2\seq V$ and $|V\backslash g\inv AB|<\zeta^*(\delta,n)|G|\leq \frac{\alpha}{k}|G|\leq \frac{1}{k}|A|$. So $AB\inv A\inv g$ contains $xU$ for some $x\in G$. Therefore $gx\inv U=gU\inv x\inv\seq  ABA\inv$. 
\end{proof}

We point out that one could deduce a quantitative version of Theorem \ref{thm:AB} in the special case of finite \emph{abelian} groups  from existing  results along the lines of \cite[Theorem 10.1]{SanBR} (perhaps with $\zeta(\delta,n)|G|$ in condition $(i)$ replaced by $\epsilon|U|$ for some constant $\epsilon$). See also the discussion of correlations for sumsets in the finite field model at the end of \cite{SchSis}. 
On the other hand, the specialization of Theorem \ref{thm:AB} to arbitrary finite groups of bounded exponent is worth stating explicitly. This can  be obtained in a routine fashion by combining Theorem \ref{thm:AB} with Corollary \ref{cor:exponent} and making a suitable choice of $\zeta$.

\begin{theorem}\label{thm:ABbe}
Fix $r\geq 1$, $\alpha>0$ and $\zeta\colon \Z^+\to \R^+$. Let $G$ be a finite group of exponent $r$ and suppose $A,B\seq G$ are such that $|A|\geq \alpha|G|$ and $|B|\geq\alpha|G|$. Then there is a normal subgroup $H\leq G$ of index $m\leq O_{r,\alpha,\zeta}(1)$, such that:
\begin{enumerate}[$(i)$]
\item $|gH\backslash AB|<\zeta(m)|G|$ for some $g\in G$,
\item $ABA\inv$ contains a coset of $H$, and
\item $H\seq AB(AB)\inv$.
\end{enumerate}
\end{theorem}

The proof in \cite{CoBogo} of Bogolyubov's Lemma  for arbitrary finite groups was based on work of Sanders \cite{SanBS}, which actually focuses on algebraic structure in product sets of the form $A^2A^{\nv 2}$ rather than $(AA\inv)^2$. Through a technical case analysis elaborating on Sanders' methods, it was shown in \cite{CoBogo} that all sensible four-product permutations among $A$ and $A\inv$ can be worked with simultaneously. Theorem \ref{thm:AB} allows us to recover this  for free. 

\begin{corollary}
Fix $\alpha>0$. Let $G$ be a finite group and suppose $A\seq G$ is such that $|A|\geq\alpha|G|$. Then there is a  $(\delta,\U(n))$-Bohr neighborhood $U\seq G$, with $\delta\inv, n\leq O_{\alpha}(1)$, such that $U\seq (AA\inv)^2\cap A^2A^{\nv 2}\cap (A\inv A)^2\cap A^{\nv 2}A^2$.
\end{corollary}
\begin{proof}
Let $X_1=(AA\inv)^2$, $X_2=(A\inv A)^2$, $X_3=A^2A^{\nv 2}$,  and $X_4=A^{\nv 2}A^2$. Given $1\leq i\leq 4$, apply Theorem \ref{thm:AB} (with the appropriate choice of $A$ and $B$) to obtain a $(\delta_i,\U(n_i))$-Bohr neighborhood $U_i$ in $G$, with $\delta_i\inv,n_i\leq O_\alpha(1)$ and $U_i\seq X_i$. 

Now note that in general, $\U(m)\times \U(n)$ isometrically embeds in $\U(m+n)$ via $(x,y)\mapsto\big(\begin{smallmatrix}x & 0\\ 0 & y\end{smallmatrix}\big)$. From this it follows that $\bigcap_i U_i$ contains a  $(\delta,\U(n))$-Bohr neighborhood in $G$, where $\delta=\min_i\delta_i$ and $n=\sum_in_i$.
\end{proof}

Finally, we turn to work of Gowers \cite{GowQRG} on \emph{quasirandom} groups which, roughly speaking, can be described as finite groups with no nontrivial (unitary) representations of small dimension. Using Theorem \ref{thm:AB}, we obtain a quick proof of one of the main results in \cite{GowQRG},  although without quantitative bounds (see also \cite[Corollary 1]{NikPyDJ} and surrounding remarks).

\begin{corollary}\label{cor:Gow}
For any $\alpha>0$ there is some $d=O_\alpha(1)$ such that if $G$ is a finite group with no nontrivial  representations of dimension less than $d$, and $A,B,C\seq G$ are such that $|A|,|B|,|C|\geq\alpha|G|$, then $|AB|>(1-\alpha)|G|$ and  $G=ABC$. 
\end{corollary}
\begin{proof}
Let $d$ be an integer greater than the $O_{\alpha,\zeta}(1)$ bound in Theorem \ref{thm:AB} where $\zeta\equiv\alpha$ (so $d=O_\alpha(1))$. Let $G$ and $A,B,C$ be as above.  Applying Theorem \ref{thm:AB}, we obtain a $(\delta,\U(n))$-Bohr neighborhood $U$ in $G$, with $\delta\inv, n<d$, and some $g\in G$ such that $|gU\backslash AB|<\alpha|G|$. Since $n<d$, we have $U=G$ by assumption on $G$. So $gU=G$, and hence $|AB|>(1-\alpha)|G|$. Now let $x\in G$ be arbitrary. Then $|xC\inv|=|C|\geq\alpha|G|$, and so $AB\cap xC\inv\neq\emptyset$, i.e., $x\in ABC$.
\end{proof}

We close this subsection with some discussion of the restriction to finite groups in Theorem \ref{thm:AB}. 
In particular, the proof of Claim 1 uses this assumption in a crucial way. To see this, let $G$ be an amenable group with left-invariant measure $\mu$, and fix $A,B\seq G$ with $\mu(A),\mu(B)\geq\alpha$. If $G$ is finite (so $\mu$ is the normalized counting measure) then, as previously observed, we have $\int \boldsymbol{1}_A\ast \boldsymbol{1}_B\,d \mu=\mu(A)\mu(B)$, which can be viewed as simply an application of Fubini's Theorem for $\mu$.  This was used  to conclude that if $S=\{x\in G:(\boldsymbol{1}_A\ast \boldsymbol{1}_B)(x)>\frac{1}{2}\alpha^2\}$, then $\mu(S)\geq\frac{1}{2}\alpha^2$ when $G$ is finite. But for infinite $G$, one cannot necessarily apply Fubini since $\mu$ is  only finitely additive, and indeed these arguments fail drastically. For example, in $(\Z,+)$ the convolution $\boldsymbol{1}_{\N}\ast\boldsymbol{1}_{\N}$ is  identically $0$ for any choice of $\mu$, hence $S=\emptyset$ (and of course we can choose $\mu$ with $\mu(\N)=1$). Compare this to Claim 1 in the proof of Theorem \ref{thm:strongBL}, which shows that in the very special case when $B=A\inv$, we have $\mu(S)\geq\frac{1}{2}\alpha$ even when $G$ is infinite. 

Despite these obstructions, there is precedent for some form of Theorem \ref{thm:AB} for arbitrary amenable groups. In particular, Jin's Theorem \cite{JinAB} says that if $A,B\seq\Z$ have positive upper Banach density then there is a syndetic set $X\seq\Z$ such that $X\backslash (A+B)$ is not syndetic.\footnote{Recall that a subset $X$ of a group $G$ is \emph{syndetic} if $G=FX$ for some finite $F\seq G$. If $G$ is amenable then syndetic is equivalent to positive \emph{lower} Banach density.} In \cite{BFW}, Bergelson, Furstenberg, and Weiss show that one may take $X$ to be a translate of a Bohr neighborhood. This  was subsequently generalized to any countable amenable group by Beiglb\"{o}ck, Bergelson, and Fish \cite{BBF}.  Thus one has an analogue of condition $(i)$ in  Theorem \ref{thm:AB} for a single countable amenable group without uniform bounds, but where $A$ and $B$ are allowed to have positive measure with respect to possibly different measures.

\subsection{The Croot-Sisask Lemma} \label{sec:CS}
In this subsection, we prove a non-quantitative analogue of the Croot-Sisask Lemma \cite{CrSi} for convolutions of arbitrary functions on amenable groups (see Theorem \ref{thm:CS}). This lemma can in fact can refer to any one of a family of results by Croot and Sisask \cite{CrSi}, as well as further variations such as in \cite{CLS} and, notably, in \cite{SanBR} where Sanders proves the Bogolyubov-Ruzsa Lemma in abelian groups with quasi-polynomial bounds.  
As usual, our result will not be quantitative, but it will be rather general in other ways (see the discussion following Corollary \ref{cor:CS}) and will again illustrate the substantial connection between stable regularity and  existing literature. 
From a  general perspective, this  also highlights the historical thread from Grothendieck's \cite{GroWAP} work on weakly almost periodic functions to ``almost periodicity" results in additive combinatorics.

We will first use Theorem \ref{thm:SARF} to prove a result for  stable functions, and then specialize to convolutions in the subsequent corollary. Roughly speaking, our theorem will say that a stable function on an amenable group is almost invariant under shifts by elements from a large unitary Bohr neighborhood. Here ``almost invariance" will be measured using the $\ell_p$-norm for some $p\geq 1$ (defined below).

For the rest of this subsection, we fix an amenable group $G$ and left-invariant measure $\mu$. 
Let $f\colon G\to [\nv 1,1]$ be a function. Given $t\in G$, we write $f_t$ for the shifted function $f(t\cdot x)$. Given some $p\geq 1$, define $\|f\|_p=(\int |f|^p\,d\mu)^{1/p}$.

\begin{theorem}\label{thm:CS}
Fix $p\geq 1$, $\epsilon>0$, and $\bk\colon\R^+\to\Z^+$. Suppose $f\colon G\to [\nv 1,1]$ is a $\bk$-stable function. Then  there is a $(\delta,\U(n))$-Bohr neighborhood $B\seq G$, with $\delta\inv,n\leq O_{\bk,\epsilon,p}(1)$ such that $\|f_t-f\|_p<\epsilon$ for any $t\in B$.
\end{theorem}
\begin{proof}
Fix $\epsilon>0$ and $p\geq 1$. Set $\epsilon'=\epsilon/2^{1/p}$ and define $\zeta\colon \R^+\times\N\to\R^+$ so that
\[
\zeta(\delta,n)=\frac{\epsilon^p}{4}\left(\frac{\delta}{2c}\right)^{n^2},
\]
where $c$ is the absolute constant from Proposition \ref{prop:Bohr}$(a)$.
Apply Theorem \ref{thm:SARF} to $f$, with parameters $\epsilon'$, $\zeta$, and $\bk$, to get a $(\delta,\U(n))$-Bohr neighborhood $C\seq G$, with $\delta\inv, n\leq O_{\epsilon,p,\bk}(1)$, such that $f$ is $\zeta(\delta,n)$-almost $\epsilon'$-constant on all translates of $C$. Let $B\seq C$ be the $(\delta/2,\U(n))$-Bohr neighborhood in $G$ obtained from the same homomorphism to $\U(n)$ defining $C$.  Using Proposition \ref{prop:Bohr}$(a)$, fix $F\seq G$ of size at most $(2c/\delta)^{n^2}$ such that $G=FB$. We fix $t\in B$ and show $\|f_t-f\|_p<\epsilon$.

Given $g\in G$, let $Z_g\seq gC$ be a set such that $\mu(Z_g)<\zeta(n,\delta)$ and $f$ is $\epsilon'$-constant on $gC\backslash Z_g$. Then set $Y_g=Z_g\cup t\inv Z_g$. Note that $\mu(Y_g)<2\zeta(n,\delta)$. Let $Y=\bigcup_{g\in F} Y_g$. Then $G\backslash Y\seq \bigcup_{g\in F}gB\backslash Y_g$. So for each $g\in F$, we can choose some set $X_g\seq gB\backslash Y_g$ so that $\{X_g\}_{g\in F}$ is a partition of $G\backslash Y$. 

Now let $h=|f_t-f|$. Then 
\[
\| f_t-f\|_p^p= \int_Y h^p\,d\mu+\sum_{g\in F}\int_{X_g} h^p\,d\mu<2\zeta(n,\delta)|F|+\sum_{g\in F}\int_{X_g}h^p\,d\mu.\tag{$\dagger$}
\]

Next fix $g\in F$ and $x\in X_g$. Then $x\in gB\backslash Y_g\seq gB\backslash Z_g\seq gC\backslash Z_g$. Moreover, we have $tx\in BgB=gB^2\seq gC$. Since $x\in  gB\backslash Y_g\seq gB\backslash t\inv Z_g$, we also have $tx\not\in Z_g$. Thus $x,tx\in gC\backslash Z_g$. Since $f$ is $\epsilon'$-constant on $gC\backslash Z_g$, we conclude that $h(x)\leq\epsilon'$. Altogether, 
\[
\text{ for any $g\in F$, }~\int_{X_g}h^p\,d\mu\leq(\epsilon')^p\mu(X_g)=\frac{\epsilon^p}{2}\mu(X_g).\tag{$\dagger\dagger$}
\]

Finally, by combining $(\dagger)$, $(\dagger\dagger)$, the  choice of $\zeta$, and the fact that $\{X_g\}_{g\in F}$ is a collection of pairwise disjoint sets, we obtain
\[
\|f_t-f\|_p^p<2|F|\zeta(n,\delta)+\frac{\epsilon^p}{2}\sum_{g\in F}\mu(X_g)\leq 2\zeta(n,\delta)\left(\frac{2c}{\delta}\right)^{n^2}+\frac{\epsilon^p}{2}=\epsilon^p.
\]
So $\|f_t-f\|_p<\epsilon$, as desired.
\end{proof}

Combined with Corollary \ref{cor:conv-stable}, we obtain a non-quantitative variation of the Croot-Sisask Lemma for convolutions of bounded functions on amenable groups.

\begin{corollary}\label{cor:CS}
Fix $p\geq 1$ and $\epsilon>0$. Suppose $f,g\colon G\to [\nv 1,1]$ are arbitrary functions. Then  there is a $(\delta,\U(n))$-Bohr neighborhood $B\seq G$, with $\delta\inv,n\leq O_{\epsilon,p}(1)$ such that $\|(f\ast g)_t-f\ast g\|_p<\epsilon$ for any $t\in B$.
\end{corollary}

We briefly compare and contrast the above statement to existing versions of the Croot-Sisask Lemma such as the original results of Croot and Sisask in \cite[Section 1.2]{CrSi} as well as the variation by Sanders \cite[Lemma 4.3]{SanBR}. First, the interest in arbitrary $\ell_p$-norms often comes from applications in which one makes a clever choice of $p$ in order to obtain better  bounds (as in \cite{SanBR}). So this leverage is lost in our result since it is non-quantitative. That being said, existing quantitative results in the \emph{noncommutative} case are in terms of shifts from a symmetric set, rather than a Bohr neighborhood. Indeed, the ability to work with Bohr neighborhoods in arbitrary groups is usually indicative of model-theoretic methods involving ultraproducts. Our result also applies to convolutions of arbitrary functions, whereas existing results typically assume one of the functions is the (balanced) indicator function of a set. On the other hand, many existing results deal with the more general ``local" regime of sets with small tripling, while our result would only be meaningful for indicator functions of dense sets (albeit in arbitrary amenable groups). Finally, we point out that the results of \cite{CrSi,SanBR} provide a sharper error estimate of the form $\epsilon\|f\|_p$, and it is not immediately obvious whether that can be recovered with our methods. However,  constant $\epsilon$ error can still be useful in applications, for example see Lemma 4.2 in \cite{Lov-surv} (Lovett's survey of \cite{SanBR}  in the case of $\mathbb{F}_2^n$).

\begin{remark}
Using the  corollaries in Section \ref{sec:main} in place of Theorem \ref{thm:SARF}, one can specialize Theorem \ref{thm:CS} and Corollary \ref{cor:CS} to certain subclasses of amenable groups and replace unitary Bohr neighborhoods with other objects. In particular, one obtains subgroups when restricting amenable groups of bounded exponent,  $(\delta,\T(n))$-Bohr neighborhoods when restricting to abelian groups, and $(\delta,\T(n),m)$-Bohr neighborhoods when restricting to amenable torsion groups. 
\end{remark}

\section{Further model-theoretic results}\label{sec:more}

We return to the model-theoretic setting of Section \ref{sec:mainMT}.
Let $T$ be a continuous theory with a sort $G$ expanding a group. Fix a left-invariant  formula $\varphi(x,z)$, with $x$ of sort $G$.\footnote{Here we mean that $\varphi(x,z)$ is left-invariant when evaluated in any model of $T$.}  Assume that $\varphi(x,z)$ is stable in $T$. Let $\varphi^\sharp(x;y,z)=\varphi(x\cdot y,z)$. 

\subsection{Summary} \label{sec:summary}
We first recall what was established in Section \ref{sec:mainMT} (with some slight changes in notation).

\begin{theorem}\label{thm:summary}
Let $M$ be a model of $T$.
\begin{enumerate}[$(a)$]
\item The Ellis semigroup of $S_\varphi(M)$ is  $(S_{\varphi^\sharp}(M),\ast)$ (defined in Lemma \ref{lem:ESid}).
\item $S_\varphi(M)$ and $S_{\varphi^\sharp}(M)$ are both weakly almost periodic, uniquely ergodic, and have unique minimal subflows, denoted  $\Gen{\varphi}{(M)}$ and $\Gen{\varphi^\sharp}{(M)}$, respectively. 
\item  $(\Gen{\varphi^\sharp}{(M)},\ast)$ is a compact  group.
\item  $\Gen{\varphi}{(M)}$ is a compact homogeneous $\Gen{\varphi^\sharp}{(M)}$-space. 
\end{enumerate}
\end{theorem}

Part $(d)$ wasn't proved explicitly, but follows from parts $(a)$, $(b)$, and $(c)$ through purely topological means  (see \cite[Lemma 4.10]{CoLSGT}).

Following standard terminology, we will refer to elements of $\Gen{\varphi}{(M)}$ (resp., $\Gen{\varphi^\sharp}{(M)}$) as \textbf{generic} $\varphi$-types (resp., $\varphi^\sharp$-types) over $M$. This follows the use of ``generic" in the general setting of topological dynamics, e.g., as in \cite{NewTD}. In particular, if $S$ is any $G$-flow, then a point $p\in S$ is generic if and only if $S=GU$ for any open neighborhood $U$ of $p$ (and hence $S=FU$ for some finite $F\seq G$ by compactness). See also \cite[Proposition 2.7]{CoLSGT}. Consequently, our use of ``generic" is also consistent with the work of Ben Yaacov in \cite[Section 6.1]{BYgroups}, which we will use below.

The main goal of this section is to establish results  on the relationship between generic types, stabilizers, and (model-theoretic) connected components. In particular, we will prove the following theorem.

\begin{theorem}\label{thm:generics}
Let $M\models T$ be sufficiently saturated, and set $G=G(M)$. Let $u$ be the identity in $\Gen{\varphi^\sharp}{(M)}$, and let $u_\varphi$ denote $u|_\varphi$ (which is in $\Gen{\varphi}{(M)}$).
\begin{enumerate}[$(a)$]
\item $G$ acts transitively on $\Gen{\varphi}{(M)}$ and $\Gen{\varphi^\sharp}{(M)}$.

\item There is a smallest $\varphi$-type-definable subgroup of $G$, denoted $G^{00}_{\varphi}$.
Moreover, $G^{00}_{\varphi}=\Stab(u_\varphi)$ and $G/G^{00}_\varphi\cong \Gen{\varphi}{(M)}$ (as topological homogeneous $G$-spaces) via the map $aG^{00}_{\varphi}\mapsto au_\varphi$. 

\item There is a smallest $\varphi^\sharp$-type-definable subgroup of $G$, denoted $G^{00}_{\varphi^\sharp}$. Moreover, $G^{00}_{\varphi^\sharp}$ is normal, $G^{00}_{\varphi^\sharp}=\Stab(u)$, and $G/G^{00}_{\varphi^\sharp}\cong \Gen{\varphi^\sharp}{(M)}$ (as topological groups) via the map $aG^{00}_{\varphi^\sharp}\mapsto au$.

\item If $a\in G$ then $au_\varphi$ is the unique type in $\Gen{\varphi}{(M)}$ concentrating on $aG^{00}_{\varphi}$, and $au$ is the unique type in $\Gen{\varphi^\sharp}{(M)}$ concentrating on $aG^{00}_{\varphi^\sharp}$.

\item $G^{00}_{\varphi^\sharp}$ is the intersection of all conjugates of $G^{00}_{\varphi}$. 

\item If $p\in \Gen{\varphi^\sharp}{(M)}$ then $\Stab(p)=G^{00}_{\varphi^\sharp}$; and if $p\in \Gen{\varphi}{(M)}$ then $\Stab(p)=aG^{00}_{\varphi}a\inv$ where $p\models aG^{00}_{\varphi}$. In particular, $G^{00}_{\varphi^\sharp}=\bigcap_{p\in \Gen{\varphi}{(M)}}\Stab(p)$.

\end{enumerate}
\end{theorem}

When $T$ is a classical first-order theory, $G^{00}_\varphi$  is $\varphi$-definable of finite index  (and typically denoted $G^0_\varphi$), and $M$ need not be saturated. In this case, the correspondence between generic $\varphi$-types and cosets of $G^0_\varphi$ was established by the authors and Terry in \cite{CPT} using earlier work of Hrushovski and the second author \cite{HrPiGLF}. The role of $\varphi^\sharp$ was made explicit by the authors in \cite{CPpfNIP}, and then a complete account of the above theorem (in the classical case) was done by the first author in \cite{CoLSGT}.

Recall from Remark \ref{rem:inmodel} that Theorem \ref{thm:summary} holds when $\varphi(x,z)$ is only stable in $M$ (using the same proofs). The same is true for Theorem \ref{thm:generics} in the classical case since this is the setting considered in \cite{CoLSGT}.  However,  if $T$ is continuous then the minimal subflow in $S_\varphi(M)$ need not be finite (or even profinite), and some amount of saturation is needed to make sense of the  type-definable groups in Theorem \ref{thm:generics}. Thus the the analogue of Theorem \ref{thm:generics} when $\varphi(x,z)$ is only stable in a single (non-saturated) model is not addressed here. 

Finally, we note that local stability for groups in continuous logic was developed by Ben Yaacov in \cite{BYgroups}, with a focus on the connection between forking and  genericity (following the classical case in \cite{HrPiGLF}). This work will be used in our proof of Proposition \ref{prop:absolute}. Further results on generic types and connected components in the continuous setting are  given in \cite{BYgroups}, but under a global assumption of stability.

\subsection{Absoluteness of the generic types}

As discussed above, when $T$ is a classical first-order theory,  the space $\Gen{\varphi}{(M)}$ in Theorem \ref{thm:generics} is finite (see also Remark \ref{rem:discrete-value} below) and hence the isomorphism type of $\Gen{\varphi}{(M)}$ does not depend on the choice of  model $M$. However, this fact is not necessary for the proof of Theorem \ref{thm:generics} in the classical case (at least the one given in \cite{CoLSGT}). By contrast, to prove Theorem \ref{thm:generics} for continuous $T$, we will  \emph{first} establish that $\Gen{\varphi}{(M)}$ is an invariant of the theory.

Throughout this subsection, we will use the notation $\Gen{\varphi}{(M)}(M)$ (resp., $\Gen{\varphi^\sharp}{(M)}(M)$) for the unique minimal flow in $S_\varphi(M)$ (resp., $S_{\varphi^\sharp}(M)$), where $M\models T$.

\begin{proposition}\label{prop:absolute}
Fix $M\models T$ and $N\succeq M$. Given $p\in S_\varphi(M)$, let $p^*\in S_\varphi(N)$ be the  unique $M$-definable extension of $p$. Then $\Gen{\varphi}{(N)}(N)=\{p^*:p\in \Gen{\varphi}{(M)}(M)\}$. 
\end{proposition}

\begin{proof}
First suppose  $\hat{p}\in \Gen{\varphi}{(N)}(N)$. Since $\hat{p}$ is generic, no condition in $\hat{p}$ forks over $\emptyset$ by \cite[Lemma 6.6]{BYgroups} (where ``forking" is in the sense of \cite{BYgroups}). It follows (using, e.g., \cite[Proposition 2.6]{BYgroups}) that $\hat{p}$ does not fork over $\emptyset$, and thus $\hat{p}$ is definable over any submodel of $N$ (e.g., $M$). Therefore   $\hat{p}=p^*$ where $p=\hat{p}|_M\in \Gen{\varphi}{(M)}(M)$. 

The converse direction is similar to \cite[Proposition 6.8]{BYgroups}, which makes a global assumption of stability. We sketch the proof in the local case. Fix $p\in \Gen{\varphi}{(M)}(M)$. 
Without loss of generality, assume $N$ is sufficiently saturated (relative to $M$). Fix some element of $\Gen{\varphi}{(N)}(N)$ which, by the above, we may write as $q^*$ for some $q\in \Gen{\varphi}{(M)}(M)$. Since $p$ is in the $G(M)$-orbit closure of $q$, it follows by saturation that there is some $g\in G(N)$ such that $(gq^*)|_M=p$. Therefore $gq^*=p^*$ by the argument above (applied with $\hat{p}=gq^*$). So $p^*\in\Gen{\varphi}{(N)}(N)$, as desired.
\end{proof}

The desired absoluteness result for generic $\varphi$-types now follows.

\begin{theorem}\label{thm:absolute}
If $M,N\models T$ then
$\Gen{\varphi}{(M)}(M)$ is homeomorphic to $\Gen{\varphi}{(N)}(N)$.  In particular, if $M\preceq N$ then the restriction map is a homeomorphism from $\Gen{\varphi}{(M)}(N)$ to $\Gen{\varphi}{(N)}(M)$  whose inverse is given by definitional extension. 
\end{theorem}

We also have the following corollary of the proof of Proposition \ref{prop:absolute}.

\begin{corollary}\label{cor:transitive}
If $M\models T$ is $\omega$-saturated, then $G(M)$ acts transitively on $\Gen{\varphi}{(M)}(M)$.
\end{corollary}
\begin{proof}
Fix $p,q\in\Gen{\varphi}{(M)}(M)$. By definability of $p$ and $q$, we have  formulas $\psi_1(y,z)=\varphi(y\cdot p,z)$ and $\psi_2(z)=\varphi(q,z)$ over $M$. Now let $N\succ M$ be sufficiently saturated (relative to $M$), and let $p^*,q^*\in S_\varphi(N)$ be the $M$-definable extensions of $p$ and $q$, respectively. As in the proof of Proposition \ref{prop:absolute}, there is some $g\in G(N)$ such that $gp^*=q^*$.
So $N\models \inf_y\sup_{z}|\psi_1(y,z)-\psi_2(z)|=0$. Since $M$ is $\omega$-saturated, there is $g\in G(M)$ such that $M\models \sup_{z}|\psi_1(g,z)-\psi_2(z)|=0$, i.e., $gp=q$.
\end{proof}

Next we establish absoluteness of the compact  group $\Gen{\varphi^\sharp}{(M)}(M)$.\footnote{At this point it is worth reminding the reader that $\varphi^\sharp(x;y,z)$ need not be stable.}  Most of the legwork will be handled by the following general topological remark.

\begin{remark}\label{rem:groupflow}
Suppose $S$ is a weakly almost periodic $G$-flow with a dense $G$-orbit, say given by $x\in S$. Let $E$ be  the Ellis semigroup of $S$, and let $K$ be the unique minimal subflow of $E$.  Then one can show that $C\coloneqq\{\sigma(x):\sigma\in K\}$ is the unique minimal subflow of $S$, and the Ellis semigroup $E(C)$ is a compact group isomorphic to $K$. In particular, it is not hard to check that any function in $K$ maps $C$ to $C$ (in fact, any function in $K$ maps all of $S$ to $C$ by \cite[Proposition II.5]{EllNer}). So we obtain a well-defined map from $K$ to $C^C$ sending $\sigma\in K$ to $\sigma{\upharpoonright}C$. This map is the desired isomorphism between $K$ and $E(C)$. 
\end{remark}

\begin{theorem}\label{thm:absolute-sharp}
If $M,N\models T$ then 
$\Gen{\varphi^\sharp}{(M)}(M)\cong \Gen{\varphi^\sharp}{(N)}(N)$ (as topological groups). In particular, if $M\preceq N$ then the restriction map is a topological group isomorphism from $\Gen{\varphi^\sharp}{(M)}(N)$ to $\Gen{\varphi^\sharp}{(N)}(M)$.
\end{theorem}
\begin{proof}
It suffices to assume $N\succeq M$. To ease notation and line up with Remark \ref{rem:groupflow}, we  let $K=\Gen{\varphi^\sharp}{(M)}(M)$, $K^*=\Gen{\varphi^\sharp}{(N)}(N)$, $C=\Gen{\varphi}{(M)}(M)$, and $C^*=\Gen{\varphi}{(N)}(N)$. Applying Remark \ref{rem:groupflow}, we have that $C$ is a $G(M)$-flow with $E(C)\cong K$, and $C^*$ is a $G(N)$-flow with $E(C^*)\cong K^*$. But note  that $C^*$ is also a $G(M)$-flow, whose Ellis semigroup we denote by $E_M(C^*)$. By Theorem \ref{thm:absolute}, restriction  from $S_{\varphi}(N)$ to $S_{\varphi}(M)$ induces a homeomorphism $\rho\colon C^*\to C$. Also, if $g\in G(M)$ and $p\in S_{\varphi}(N)$ then $(gp)|_M=g(p|_M)$. Therefore, $\rho$ is a $G(M)$-flow isomorphism, and thus canonically induces an isomorphism (of $G(M)$-flows and topological groups) from $E_M(C^*)$ to $E(C)$. To summarize, we have  isomorphisms $K^*\cong E(C^*)$ and $E_M(C^*)\cong E(C)\cong K$. So to establish $K\cong K^*$, it suffices to show  $E(C^*)=E_M(C^*)$. 

For $g\in G(N)$, let $\breve{g}\colon C^*\to C^*$ denote the action by $g$. We fix $g\in G(N)$ and show that $\breve{g}\in E_M(C^*)=\overline{\{\breve{a}:a\in G(M)\}}$. In light of the homeomorphism $\rho$, it suffices to fix $b_1,\ldots,b_n\in M^z$, $p_1,\ldots,p_n\in C^*$, and  $\epsilon>0$, and find  $a\in G(M)$ such that $|\varphi(gp_i,b_i)-\varphi(ap_i,b_i)|<\epsilon$ for all $1\leq i \leq n$. Such an  $a$ exists since $M\preceq N$ and each $p_i$ is definable over $M$ (by Proposition \ref{prop:absolute}).

We have now shown that $K^*$ and $K$ are isomorphic as topological groups. By unraveling all of the maps used above, one can check that the underlying isomorphism  is precisely the restriction map. 
\end{proof}

\subsection{Proof of Theorem \ref{thm:generics}}

Throughout this section we let $M\models T$ be a sufficiently saturated model, and set $G=G(M)$. We  follow the notation laid out in Section \ref{sec:summary} and in the statement of Theorem \ref{thm:generics}. We will frequently use the fact that the identity $u$ in $\Gen{\varphi^\sharp}{(M)}$ commutes with every element of $S_{\varphi^\sharp}(M)$ (see Corollary \ref{cor:EN}$(d)$). In particular, $gu=ug$ for all $g\in G$. 

Recall that for a point $p$ in a $G$-flow $S$, the stabilizer $\Stab(p)=\{g\in G:gp=p\}$ is a subgroup of $G$. Note also that $\Stab(gp)=g\Stab(p)g\inv$.

\begin{proposition}\label{prop:generics}$~$
\begin{enumerate}[$(a)$]
\item $\Stab(u)$ is normal and $\varphi^\sharp$-type-definable of bounded index, and $G/\Stab(u)\cong \Gen{\varphi^\sharp}{(M)}$ (as topological groups) via  $g\Stab(u)\mapsto gu$.
\item $\Stab(u_\varphi)$ is $\varphi$-type-definable of bounded index, and $G/\Stab(u_\varphi)\cong \Gen{\varphi}{(M)}$ (as topological homogeneous $G$-spaces) via  $g\Stab(u_\varphi)\mapsto gu_\varphi$.
\item $u\models \Stab(u)$ and $u_\varphi\models \Stab(u_\varphi)$. 
\end{enumerate}
\end{proposition}
\begin{proof}
Part $(a)$. Define $\pi\colon G\to \Gen{\varphi^\sharp}{(M)}$ so that $\pi(g)=gu$. By Proposition \ref{prop:pi-ext}, $\pi$ is a $\varphi^\sharp$-definable homomorphism with dense image. Since $\Gen{\varphi^\sharp}{(M)}$ is bounded (by Theorem \ref{thm:absolute-sharp}) and $M$ is sufficiently saturated, $\pi$ induces an isomorphism of topological groups from $G/\ker(\pi)$ to $\Gen{\varphi^\sharp}{(M)}$. Finally, note that $\ker(\pi)=\Stab(u)$.

Part $(b)$. Since $G$ acts transitively on $\Gen{\varphi}{(M)}$ (by Corollary \ref{cor:transitive}), it makes sense to view $\Gen{\varphi}{(M)}$ as a homogeneous $G$-space.
To prove the result, it suffices to show that the map $g\mapsto gu_\varphi$ is $\varphi$-definable, since then we can argue as in part $(a)$ (using Theorem \ref{thm:absolute} in place of Theorem \ref{thm:absolute-sharp}). Consider the map $\sigma_u\in E(S_\varphi(M))$ from the proof of Lemma \ref{lem:ESid}. Recall from Remark \ref{rem:SWAP} that $\sigma_u$ is continuous.  One can also check that $\sigma_u(p|_\varphi)=(p\ast u)|_\varphi$ for any  $p\in S_{\varphi^\sharp}(M)$.
Hence $\sigma_u$ maps $S_\varphi(M)$ to $\Gen{\varphi}{(M)}$ (recall Corollary \ref{cor:EN}$(b)$) and extends the map $g\mapsto gu_\varphi$.

Part $(c)$. By the proof of part $(a)$ (and Proposition \ref{prop:pi-ext}),  a type $p\in S_{\varphi^\sharp}(M)$ concentrates on $\Stab(u)$ if and only if $p\ast u=u$. So $u\models \Stab(u)$. 
Similarly, from the proof of part $(b)$, we see that a type $p\in S_\varphi(M)$ concentrates on $\Stab(u_\varphi)$ if and only if $\sigma_u(p)=u_\varphi$. Moreover, $\sigma_u(u_\varphi)=(u\ast u)|_\varphi=u_\varphi$. So $u_\varphi\models \Stab(u_\varphi)$.
\end{proof}

Note that the proof of Proposition \ref{prop:generics}$(a)$ establishes surjectivity of the map $g\mapsto gu$ from $G$ to $\Gen{\varphi^\sharp}{(M)}$. So we also have the following conclusion.

\begin{corollary}\label{cor:transitive-sharp}
$G$ acts transitively on $\Gen{\varphi^\sharp}{(M)}$.
\end{corollary}

To finish the proof of Theorem \ref{thm:generics}, we will need the following basic fact, whose proof is left as an exercise (and is a routine adaptation of the classical case).

\begin{fact}\label{fact:extend-generic}
Let $\psi(x,w)$ be a left-invariant formula (with $x$ of sort $G$), and assume that $S_\psi(M)$ has a unique minimal subflow (i.e., there is a generic $\psi$-type over $M$). Let $H\leq G$ be a $\psi$-type-definable bounded-index subgroup of $G$. Then there is some generic $p\in S_\psi(M)$ concentrating on $H$.
\end{fact}

\begin{corollary}
$~$
\begin{enumerate}[$(a)$]
\item $\Stab(u)$ is the smallest $\varphi^\sharp$-type-definable bounded-index subgroup of $G$.
\item  $\Stab(u_\varphi)$ is the smallest $\varphi$-type-definable bounded-index subgroup of $G$.
\end{enumerate}
\end{corollary}
\begin{proof}
Part $(a)$. Let $H$ be a $\varphi^\sharp$-type-definable bounded-index subgroup of $G$. By Corollary \ref{cor:transitive-sharp} and Fact \ref{fact:extend-generic} there is some $g\in G$ such that $gu\models H$. Now if $a\in \Stab(u)$ then $agu=aug=ug=gu$, so $agu\models H$. Since $agu\models aH$, we conclude $H=aH$, i.e., $a\in H$. 

Part $(b)$. Let $H$ be a $\varphi$-type-definable bounded-index subgroup of $G$. Then $\Stab(u)\seq H$ by part $(a)$, so $u\models H$ by Proposition \ref{prop:generics}$(c)$. Since $H$ is $\varphi$-type-definable, we have $u_\varphi\models H$. So if $g\in \Stab(u_\varphi)$, then $u_\varphi$ also concentrates on $gH$, hence $H=gH$, i.e., $g\in H$. 
\end{proof}

The previous corollary justifies writing $G^{00}_{\varphi^\sharp}$ and $G^{00}_{\varphi}$ for $\Stab(u)$ and $\Stab(u_\varphi)$, respectively. Combining the above results, we now have proved parts $(a)$, $(b)$, and $(c)$ of Theorem \ref{thm:generics}, which leaves parts $(d)$, $(e)$, and $(f)$.  

For part $(d)$, fix some $a\in G$.  Then $au\models aG^{00}_{\varphi^\sharp}$  by Proposition \ref{prop:generics}$(c)$. Conversely, suppose $p\in\Gen{\varphi^\sharp}{(M)}$ concentrates on $aG^{00}_{\varphi^\sharp}$. Fix $g\in G$ such that $p=gu$. Then  $p\models gG^{00}_{\varphi^\sharp}$, thus $aG^{00}_{\varphi^\sharp}=gG^{00}_{\varphi^\sharp}$, i.e., $au=gu=p$. By a similar argument, $au_\varphi$ is the unique type in $\Gen{\varphi}{(M)}$ concentrating on $aG^{00}_{\varphi}$. 

For part $(e)$, suppose $g\not\in G^{00}_{\varphi^\sharp}$. Then $gu\neq u$, so there is some instance $\varphi^\sharp(x;a,b)$ such that $\varphi^\sharp(u;a,b)\neq \varphi^\sharp(gu;a,b)$. Since $ua=au$, it  follows that $\varphi(au_\varphi,b)\neq \varphi(gau_\varphi,b)$. So $g\not\in \Stab(au_\varphi)=aG^{00}_{\varphi}a\inv$.

Part $(f)$ now follows immediately from the previous statements, which finishes the proof of Theorem \ref{thm:generics}.

\begin{remark}\label{rem:discrete-value}
For the sake of completeness, we sketch how to recover finiteness of $\Gen{\varphi}{(M)}$ when  the topology on $S_\varphi(M)$ induced from the canonical local metric is discrete (see \cite[Definition 6.1]{BYU}). Let $\mu$ be the unique left-invariant Keisler $\varphi$-functional over $M$. Then, assuming discreteness of $S_\varphi(M)$,  we can write $\mu$ as a weighted sum of  Dirac measures, i.e., $\mu=\sum_{i\in I}\alpha_i p_i$ where $\alpha_i>0$ and $p_i\in S_\varphi(M)$. This can be proved as in case of classical  logic, but also follows immediately from \cite[Theorem 3.12]{CCP} which gives an analogous result for any local Keisler functional defined from a continuous stable formula. Now, by Proposition \ref{prop:domination} and the relationship between the Haar functionals on $\Gen{\varphi^\sharp}{(M)}$ and on $\Gen{\varphi}{(M)}$ (recall Theorem \ref{thm:summary}$(d)$), each $p_i$ above is in $\Gen{\varphi}{(M)}$ and is a point of positive Haar measure. But it is a standard exercise that any compact homogeneous space with such a point must be finite. 
\end{remark}


\begin{thebibliography}{10}

\bibitem{AGG}
M.~A. Alekseev, L.~Y. Glebski\u\i, and E.~I. Gordon, \emph{On approximations of
  groups, group actions and {H}opf algebras}, Zap. Nauchn. Sem. S.-Peterburg.
  Otdel. Mat. Inst. Steklov. (POMI) \textbf{256} (1999), no.~Teor. Predst. Din.
  Sist. Komb. i Algoritm. Metody. 3, 224--262, 268. \MR{1708567}

\bibitem{BBF}
M. Beiglb\"ock, V. Bergelson, and A. Fish, \emph{Sumset phenomenon in countable
  amenable groups}, Adv. Math. \textbf{223} (2010), no.~2, 416--432.
  \MR{2565535}
  
  \bibitem{BYtop}
I. Ben~Yaacov, \emph{Topometric spaces and perturbations of metric structures},
  Log. Anal. \textbf{1} (2008), no.~3-4, 235--272. \MR{2448260}

\bibitem{BYgroups}
\bysame, \emph{Stability and stable groups in continuous logic}, J.
  Symbolic Logic \textbf{75} (2010), no.~3, 1111--1136. \MR{2723787}

\bibitem{BYGro}
\bysame, \emph{Model theoretic stability and definability of types, after
  {A}. {G}rothendieck}, Bull. Symb. Log. \textbf{20} (2014), no.~4, 491--496.
  \MR{3294276}

\bibitem{BBHU}
I. Ben~Yaacov, A. Berenstein, C.~W. Henson, and A. Usvyatsov, \emph{Model
  theory for metric structures}, Model theory with applications to algebra and
  analysis. {V}ol. 2, London Math. Soc. Lecture Note Ser., vol. 350, Cambridge
  Univ. Press, Cambridge, 2008, pp.~315--427. \MR{2436146 (2009j:03061)}
  
 \bibitem{BYTs}
I. Ben~Yaacov and T. Tsankov, \emph{Weakly almost periodic functions,
  model-theoretic stability, and minimality of topological groups}, Trans.
  Amer. Math. Soc. \textbf{368} (2016), no.~11, 8267--8294. \MR{3546800}


\bibitem{BYU}
I. Ben~Yaacov and A. Usvyatsov, \emph{Continuous first order logic and local
  stability}, Trans. Amer. Math. Soc. \textbf{362} (2010), no.~10, 5213--5259.
  \MR{2657678}
  
  \bibitem{BFW}
V. Bergelson, H. Furstenberg, and B. Weiss, \emph{Piecewise-{B}ohr sets of
  integers and combinatorial number theory}, Topics in discrete mathematics,
  Algorithms Combin., vol.~26, Springer, Berlin, 2006, pp.~13--37. \MR{2249261}

\bibitem{Bog39}
N. Bogolio\`uboff, \emph{Sur quelques propri\'et\'es arithm\'etiques des
  presque-p\'eriodes}, Ann. Chaire Phys. Math. Kiev \textbf{4} (1939),
  185--205. \MR{0020164}

\bibitem{BGT}
E. Breuillard, B. Green, and T. Tao, \emph{The structure of approximate
  groups}, Publ. Math. Inst. Hautes \'Etudes Sci. \textbf{116} (2012),
  115--221. \MR{3090256}

\bibitem{CCP}
N. Chavarria, G. Conant, and A. Pillay, \emph{Continuous stable regularity}, J.
  Lond. Math. Soc. (2) \textbf{109} (2024), no.~1, Paper No. e12822, 36.
  \MR{4680211}
  
 \bibitem{CoBogo}
G. Conant, \emph{On finite sets of small tripling or small alternation in
  arbitrary groups}, Combin. Probab. Comput. \textbf{29} (2020), no.~6,
  807--829. \MR{4173133}
  
 
\bibitem{CoLSGT}
\bysame, \emph{Stability in a group}, Groups Geom. Dyn. \textbf{15} (2021),
  no.~4, 1297--1330. \MR{4349660}

\bibitem{CoQSAR}
\bysame, \emph{Quantitative structure of stable sets in arbitrary finite
  groups}, Proc. Amer. Math. Soc. \textbf{149} (2021), no.~9, 4015--4028.
  \MR{4291597}
  
 
\bibitem{CHP}
G. Conant, E. Hrushovski, and A. Pillay, \emph{Compactifications of
  pseudofinite and pseudo-amenable groups}, arXiv:2308.08440, 2023.
  
  \bibitem{CPpfNIP}
G. Conant and A. Pillay, \emph{Pseudofinite groups and {VC}-dimension}, J.
  Math. Log. \textbf{21} (2021), no.~2, Paper No. 2150009, 23. \MR{4290498}


\bibitem{CPT}
G. Conant, A. Pillay, and C. Terry, \emph{A group version of stable
  regularity}, Math. Proc. Cambridge Philos. Soc. \textbf{168} (2020), no.~2,
  405--413. \MR{4064112}

\bibitem{CPTNIP}
\bysame, \emph{Structure and regularity for subsets of groups with finite
  {VC}-dimension}, J. Eur. Math. Soc. (JEMS) \textbf{24} (2022), no.~2,
  583--621. \MR{4382479}
  
  \bibitem{CLS}
E. Croot, I. {\L}aba, and O. Sisask, \emph{Arithmetic progressions in sumsets
  and {$L^p$}-almost-periodicity}, Combin. Probab. Comput. \textbf{22} (2013),
  no.~3, 351--365. \MR{3053851}

  
 \bibitem{CrSi}
E. Croot and O. Sisask, \emph{A probabilistic technique for finding
  almost-periods of convolutions}, Geom. Funct. Anal. \textbf{20} (2010),
  no.~6, 1367--1396. \MR{2738997}


\bibitem{EllLTD}
R. Ellis, \emph{Lectures on topological dynamics}, W. A. Benjamin, Inc., New
  York, 1969. \MR{0267561}

\bibitem{EllNer}
R. Ellis and M. Nerurkar, \emph{Weakly almost periodic flows}, Trans. Amer.
  Math. Soc. \textbf{313} (1989), no.~1, 103--119. \MR{930084}
  
 \bibitem{Glas-surv}
E. Glasner, \emph{Enveloping semigroups in topological dynamics}, Topology
  Appl. \textbf{154} (2007), no.~11, 2344--2363. \MR{2328017}
  
  \bibitem{GowQRG}
W.~T. Gowers, \emph{Quasirandom groups}, Combin. Probab. Comput. \textbf{17}
  (2008), no.~3, 363--387. \MR{2410393}

  
 \bibitem{GreenSLAG}
B. Green, \emph{A {S}zemer\'edi-type regularity lemma in abelian groups, with
  applications}, Geom. Funct. Anal. \textbf{15} (2005), no.~2, 340--376.
  \MR{2153903}


\bibitem{GroWAP}
A. Grothendieck, \emph{Crit\`eres de compacit\'{e} dans les espaces
  fonctionnels g\'{e}n\'{e}raux}, Amer. J. Math. \textbf{74} (1952), 168--186.
  \MR{0047313}
  
 \bibitem{HofMo3}
K.~H. Hofmann and S.~A. Morris, \emph{The structure of compact groups}, De
  Gruyter Studies in Mathematics, vol.~25, De Gruyter, Berlin, 2013, A primer
  for the student---a handbook for the expert, Third edition, revised and
  augmented. \MR{3114697}

\bibitem{HruAG}
E. Hrushovski, \emph{Stable group theory and approximate subgroups}, J. Amer.
  Math. Soc. \textbf{25} (2012), no.~1, 189--243. \MR{2833482}
  
 \bibitem{HrKrP1}
E. Hrushovski, K. Krupi\'{n}ski, and A. Pillay, \emph{Amenability, connected
  components, and definable actions}, Selecta Math. (N.S.) \textbf{28} (2022),
  no.~1, Paper No. 16, 56. \MR{4350207}
  
  \bibitem{HrPiGLF}
E. Hrushovski and A. Pillay, \emph{Groups definable in local fields and
  pseudo-finite fields}, Israel J. Math. \textbf{85} (1994), no.~1-3, 203--262.
  \MR{1264346}

\bibitem{JinAB}
R. Jin, \emph{The sumset phenomenon}, Proc. Amer. Math. Soc. \textbf{130}
  (2002), no.~3, 855--861. \MR{1866042}

\bibitem{Kazh}
D. Kazhdan, \emph{On {$\varepsilon $}-representations}, Israel J. Math.
  \textbf{43} (1982), no.~4, 315--323. \MR{693352}
  
  \bibitem{Kowalski-book}
E. Kowalski, \emph{An introduction to the representation theory of groups},
  Graduate Studies in Mathematics, vol. 155, American Mathematical Society,
  Providence, RI, 2014. \MR{3236265}

\bibitem{KrMau}
J.-L. Krivine and B. Maurey, \emph{Espaces de {B}anach stables}, Israel J.
  Math. \textbf{39} (1981), no.~4, 273--295. \MR{636897}
  
\bibitem{Lov-surv}
S. Lovett, \emph{An exposition of sanders' quasi-polynomial freiman-ruzsa
  theorem}, Theory of Computing (2015), 1--14.
  
 \bibitem{MaSh}
M. Malliaris and S. Shelah, \emph{Regularity lemmas for stable graphs}, Trans.
  Amer. Math. Soc. \textbf{366} (2014), no.~3, 1551--1585. \MR{3145742}

\bibitem{MOS}
S. Montenegro, A. Onshuus, and P. Simon, \emph{Stabilizers, {${\rm NTP}_2$}
  groups with f-generics, and {PRC} fields}, J. Inst. Math. Jussieu \textbf{19}
  (2020), no.~3, 821--853. \MR{4094708}
  
  \bibitem{NewTD}
L. Newelski, \emph{Topological dynamics of definable group actions}, J.
  Symbolic Logic \textbf{74} (2009), no.~1, 50--72. \MR{2499420}
  
 \bibitem{NikPyDJ}
N. Nikolov and L. Pyber, \emph{Product decompositions of quasirandom groups and
  a {J}ordan type theorem}, J. Eur. Math. Soc. (JEMS) \textbf{13} (2011),
  no.~4, 1063--1077. \MR{2800484}
  
  \bibitem{PalPFS}
D. Palac\'{\i}n, \emph{On compactifications and product-free sets}, J. Lond.
  Math. Soc. (2) \textbf{101} (2020), no.~1, 156--174. \MR{4072489}
  
 \bibitem{Patbook}
A.~L.~T. Paterson, \emph{Amenability}, Mathematical Surveys and Monographs,
  vol.~29, American Mathematical Society, Providence, RI, 1988. \MR{961261}
  
\bibitem{PiGST}
A. Pillay, \emph{Geometric stability theory}, Oxford Logic Guides, vol.~32, The
  Clarendon Press, Oxford University Press, New York, 1996, Oxford Science
  Publications. \MR{1429864}


  
 \bibitem{RZbook}
L. Ribes and P. Zalesskii, \emph{Profinite groups}, second ed., Ergebnisse der
  Mathematik und ihrer Grenzgebiete. 3. Folge. A Series of Modern Surveys in
  Mathematics [Results in Mathematics and Related Areas. 3rd Series. A Series
  of Modern Surveys in Mathematics], vol.~40, Springer-Verlag, Berlin, 2010.
  \MR{2599132}
  
  \bibitem{Ruz94}
I.~Z. Ruzsa, \emph{Generalized arithmetical progressions and sumsets}, Acta
  Math. Hungar. \textbf{65} (1994), no.~4, 379--388. \MR{1281447}

\bibitem{SanBS}
T. Sanders, \emph{On a nonabelian {B}alog-{S}zemer\'edi-type lemma}, J. Aust.
  Math. Soc. \textbf{89} (2010), no.~1, 127--132. \MR{2727067}
  
  
  \bibitem{SanBR}
\bysame, \emph{On the {B}ogolyubov-{R}uzsa lemma}, Anal. PDE \textbf{5}
  (2012), no.~3, 627--655. \MR{2994508}
  
  
    \bibitem{SchSis}
T. Schoen and O. Sisask, \emph{Roth's theorem for four variables and additive
  structures in sums of sparse sets}, Forum Math. Sigma \textbf{4} (2016),
  Paper No. e5, 28. \MR{3482282}

\bibitem{Shbook}
S. Shelah, \emph{Classification theory and the number of nonisomorphic models},
  second ed., Studies in Logic and the Foundations of Mathematics, vol.~92,
  North-Holland Publishing Co., Amsterdam, 1990. \MR{1083551 (91k:03085)}


  
  \bibitem{SzemRL}
E. Szemer\'edi, \emph{Regular partitions of graphs}, Probl\`emes combinatoires
  et th\'eorie des graphes ({C}olloq. {I}nternat. {CNRS}, {U}niv. {O}rsay,
  {O}rsay, 1976), Colloq. Internat. CNRS, vol. 260, CNRS, Paris, 1978,
  pp.~399--401. \MR{540024}
  
\bibitem{TaoICM}
T. Tao, \emph{The dichotomy between structure and randomness}, presented at the
  2006 International Congress of Mathematicians, Madrid, available at:
  \url{https://www.math.ucla.edu/~tao/preprints/Slides/icmslides2.pdf}.

\bibitem{TeWo}
C. Terry and J. Wolf, \emph{Stable arithmetic regularity in the finite field
  model}, Bull. Lond. Math. Soc. \textbf{51} (2019), no.~1, 70--88.
  \MR{3919562}

\bibitem{TeWo2}
\bysame, \emph{Quantitative structure of stable sets in finite abelian groups},
  Trans. Amer. Math. Soc. \textbf{373} (2020), no.~6, 3885--3903. \MR{4105513}

\end{thebibliography}
\end{document}